\documentclass[12pt,leqno]{amsart}
\usepackage[colorlinks,citecolor=red,linkcolor=blue,pagebackref,hypertexnames=false]{hyperref}

\makeatletter
\def\@tocline#1#2#3#4#5#6#7{\relax
  \ifnum #1>\c@tocdepth 
  \else
    \par \addpenalty\@secpenalty\addvspace{#2}%
    \begingroup \hyphenpenalty\@M
    \@ifempty{#4}{%
      \@tempdima\csname r@tocindent\number#1\endcsname\relax
    }{%
      \@tempdima#4\relax
    }%
    \parindent\z@ \leftskip#3\relax \advance\leftskip\@tempdima\relax
    \rightskip\@pnumwidth plus4em \parfillskip-\@pnumwidth
    #5\leavevmode\hskip-\@tempdima
      \ifcase #1
       \or\or \hskip 1em \or \hskip 2em \else \hskip 3em \fi%
      #6\nobreak\relax
    \dotfill\hbox to\@pnumwidth{\@tocpagenum{#7}}\par
    \nobreak
    \endgroup
  \fi}
\makeatother

\usepackage{tikz,amsthm,amsmath,amstext,amssymb,amscd,epsfig,euscript, mathrsfs, dsfont,pspicture,multicol,graphpap,graphics,graphicx,times,enumerate,subfig,sidecap,wrapfig,color,pict2e}

\usepackage{enumerate}

 \numberwithin{equation}{section}

\def\bC{{\mathbb{C}}}
\def\bD{{\mathbb{D}}}
\def\bR{{\mathbb{R}}}
\def\bS{{\mathbb{S}}}

\def\bN{{\mathbb{N}}}

\def\hm{\omega}
\def\om{\Omega}
\def\R{\mathbb{R}}

\def\cH{{\mathcal{H}}}

\def\one{\mathds{1}}

\def\ve{\varepsilon}

\renewcommand{\d}{{\partial}}
\def\lec{\lesssim}
\def\gec{\gtrsim}

\DeclareMathOperator{\diam}{diam}
\def\Tan{\mathop\mathrm{Tan}} 					
\def\dist{\mathop\mathrm{dist}} 						
\def\supp{\mathop\mathrm{supp}}					
\renewcommand{\div}{\mathop\mathrm{div }}			

\newcommand{\ps}[1]{\left( #1 \right)}

\newcommand{\ck}[1]{\left\{#1 \right\}}

\newcommand{\cnj}[1]{\overline{#1}}

\newcommand{\vvs}{\vspace{4mm}}

\def\XXint#1#2#3{{\setbox0=\hbox{$#1{#2#3}{\int}$ }
\vcenter{\hbox{$#2#3$ }}\kern-.58\wd0}}

\def\grad{\nabla}

\theoremstyle{plain}
\newtheorem{theorem}{Theorem}

\newtheorem{lemma}[theorem]{Lemma}

\newtheorem*{thmii}{Theorem II}
\newtheorem*{thmiii}{Theorem III}
\newtheorem*{lemi}{Lemma I}
\newtheorem*{lemii}{Lemma II}
\newtheorem*{cori}{Corollary I}

\theoremstyle{definition}

\newtheorem{definition}[theorem]{Definition}
\newtheorem{remark}[theorem]{Remark}
\newtheorem*{thmi}{Theorem I}

\numberwithin{equation}{section}
\numberwithin{theorem}{section}

\newcommand\eqn[1]{\eqref{e:#1}}

\newcommand\Theorem[1]{Theorem \ref{t:#1}}
\newcommand\Lemma[1]{Lemma \ref{l:#1}}

\newcommand\Definition[1]{Definition \ref{d:#1}}

  \newcommand{\F}{\mathcal{F}}
  \newcommand{\W}{\mathcal{W}}

  \DeclareFontFamily{U}{mathb}{\hyphenchar\font45} 
\DeclareFontShape{U}{mathb}{m}{n}{
      <5> <6> <7> <8> <9> <10> gen * mathb
      <10.95> mathb10 <12> <14.4> <17.28> <20.74> <24.88> mathb12
      }{}
\DeclareSymbolFont{mathb}{U}{mathb}{m}{n}

\DeclareMathSymbol{\toitself}      {3}{mathb}{"FD}  

\begin{document}

\title[Absolute continuity of harmonic measure]{Absolute continuity of harmonic measure for domains with lower regular boundaries}

\author[Murat Akman]{Murat Akman}

\address{{\bf Murat Akman}
\\
Department of Mathematics, University of Connecticut, Storrs, CT 06269-3009} \email{murat.akman "at" uconn.edu}

\author{Jonas Azzam}
\address{{\bf Jonas Azzam}
\\
School of Mathematics, University of Edinburgh, JCMB, Kings Buildings,
Mayfield Road, Edinburgh,
EH9 3JZ, Scotland.}
\email[Jonas Azzam]{j.azzam "at" ed.ac.uk}

\author[Mihalis Mourgoglou]{Mihalis Mourgoglou}

\address{{\bf Mihalis Mourgoglou}
\\
Departament de Matem\`atiques\\
 Universitat Aut\`onoma de Barcelona
\\
Edifici C Facultat de Ci\`encies
\\
08193 Bellaterra (Barcelona), Catalonia
}
\email{mourgoglou "at" mat.uab.cat}


\abstract{We study absolute continuity of harmonic measure with respect to surface measure on domains $\Omega$ that have large complements. We show that if $\Gamma\subset  \bR^{d+1}$ is Ahlfors regular and splits $ \bR^{d+1}$ into two NTA domains, then $\omega_{\Omega}\ll \cH^{d}$ on $\Gamma\cap \partial\Omega$. This result is a natural generalisation of a result of Wu in \cite{Wu86}.

We also prove that almost every point in $\Gamma\cap\partial\Omega$ is a cone point if $\Gamma$ is a Lipschitz graph. Combining these results and a result from \cite{AHMMMTV}, we characterize sets of absolute continuity with finite $\cH^{d}$-measure both in terms of the cone point condition and in terms of the rectifiable structure of the boundary. This generalizes the results of McMillan in \cite{McM69} and Pommerenke in \cite{Pom86}.

Finally, we also show our first result holds for elliptic measure associated with real second order divergence form elliptic operators with a mild assumption on the gradient of the matrix.
}}

\maketitle

\tableofcontents

\section{Introduction}

\subsection{Background}

Classifying sets of absolute continuity and singularity for harmonic measure with respect to surface measure on pieces of rough domains has been extensively studied for decades. In \cite[Theorem 1; p. 830 and p. 18 in the translation]{Lav36}, Lavrentiev constructed an example of a simply connected domain $\Omega$ in the plane and a set $E\subset \partial \Omega$ with the property that $E$ has zero linear measure and positive harmonic measure with respect to $\Omega$. 
This result was further simplified and strengthened by Carleson in \cite[Theorem. (A)]{Car73} and by McMillan and Piranian in \cite[Theorem 1]{MP73}. Considering this example, it was natural to consider what extra criteria were necessary for absolute continuity to occur. 

McMillan showed in \cite[Theorem 2]{McM69} that for bounded simply connected domains $\Omega\subset \bC$, $\omega_{\Omega}\ll \cH^{1}\ll \omega_{\Omega}$ on the set of cone points, and Pommerenke would later demonstrate in \cite[Corollary 2]{Pom86} that in fact harmonic measure is supported on either the cone points or a set of zero length. This implies that if $\omega_{\Omega}\ll \cH^{1}$ on some subset $E\subset \d\Omega$, then $\omega_{\Omega}$-almost each of those points must be a cone point.

There are also many results that give sufficient conditions for absolute continuity in terms of the geometry of the boundary rather than the geometry of the interior of the domain. It was shown by {\O}ksendal in \cite[p. 471]{Ok80} that if $L$ is a line and $\Omega\subset\mathbb{R}^{2}$ is a simply connected domain and if $E\subset \partial \Omega\cap L$ is a set with vanishing $\cH^{1}$ measure, then $E$ has zero harmonic measure  with respect  to $\Omega$. In \cite[Theorem 3]{KWu82}, Kaufman and Wu generalized this result by showing $L$ can be replaced with a bi-Lipschitz curve. It was also observed in the same article that one cannot replace $L$ with a quasicircle; thus the finite length of this surrogate set $L$ is as important as its geometry. In fact, later Bishop and Jones showed in \cite[Theorem 1]{BJ90} that $L$ can be any curve of finite length. In other words, harmonic measure can be concentrated on set of length zero but this set must be dispersed in the plane in such a way that it is impossible to be contained in a rectifiable curve. 

Note that the set of cone points for a domain is contained in a countable union of Lipschitz graphs, so the results of Kaufmann, Wu, Bishop, and Jones show that one can have weaker conditions that imply absolute continuity. Combined with Pommerenke's theorem, however, the result of Bishop and Jones shows that if $L$ is a Lipschitz curve, then $\omega_{\Omega}$-almost every point in $L\cap \d\Omega$ is a cone point, so in fact if harmonic measure is rectifiable on a subset of the boundary, that forces the domain to be wide open around this set.

In \cite{BJ90}, Bishop and Jones also showed the following.

\begin{theorem}[{\cite[Lemma 8.1]{BJ90}}] \label{t:bj}
There is a curve $\Gamma\subset \bC$ and sets $K\subset E\subset \Gamma$ such that for all $x\in \Gamma$, $y\in E$, and $0<r<\diam \Gamma$,
\[
\cH^{1}(\Gamma\cap B(x,r))\leq C_{1} r ,\]
\[
\cH^{1}(E\cap B(y,r))\geq C_{2} r ,\]
and
\[
\omega_{E^{c}}(K)>0=\cH^{1}(K).\]
\end{theorem}

Thus, extra assumptions on the domain (like simple connectedness) are necessary as well as assumptions on the structure of $E$.

The higher dimensional version of Bishop and Jones' result fails even with an analogous of connectivity assumption. In \cite[Example, p. 485]{Wu86}, Wu constructed a topological ball $\Omega\subset \bR^{3}$ and a set $E\subset\partial \Omega\cap \bR^{2}$ so that $\mbox{dim}_{\cH}(E)=1$ (which is stronger than $\cH^{2}(E)=0$) but $\omega_{\Omega}(E)>0$. In the same article, Wu proved that, with some extra geometric assumptions on the domain, one can obtain absolute continuity:

\begin{theorem}{\cite[Theorem, p. 486]{Wu86}}
Let $\Omega\subset\mathbb{R}^{d+1}$ be a bounded connected domain satisfying the exterior corkscrew condition. Let $\Gamma$ be a topological $d$-sphere in $\mathbb{R}^{d+1}$, whose interior $\Omega_1$ and exterior $\Omega_2$ are both non-tangentially accessible domains (NTA) such that $\omega_{\Omega_{i}}\ll \cH^{d}|_{\Gamma}$ for $i=1,2$. Then $\omega_{\Omega}\ll \cH^{d}$ on $\partial\Omega\cap\Gamma$.
\label{t:wu}
\end{theorem}

For the definitions of the corkscrew condition and NTA, see \Definition{CS} and \Definition{nta} below.

Of course now it is necessary to know which NTA domains have absolutely continuous harmonic measures, since an answer to this tells us, via \Theorem{wu}, when harmonic measure for exterior corkscrew domains is absolutely continuous. There are some results giving intrinsic geometric criteria for when this happens, but it seems unlikely that there is a necessary and sufficient geometric condition. Dahlberg showed in \cite[Theorem 1]{Dahl77} that if $\Omega\subset \bR^{d+1}$ is a Lipschitz domain, then $\omega_{\Omega}\ll \cH^{d}|_{\d\Omega} \ll \omega_{\Omega}$. Later, David and Jerison in \cite[Theorem 2]{DJ90} and independently by Semmes in \cite{Sem} extended this to NTA domains with Ahlfors regular boundaries (see also \cite[Theorem 1.8]{AzzHarm} for a local version of this result). In \cite[Theorem 1.2]{Bad12}, it was shown that if $\Omega$ is an NTA domain whose boundary has locally finite $\cH^{d}$-measure, then $\cH^{d}|_{\partial\Omega}\ll\omega$, and  $\omega\ll \cH^d$ on $\Theta$, where
\begin{equation}\label{e:bad}
\Theta:=\left\{x\in\partial\Omega:\, \liminf\limits_{r\to 0} r^{-d}\cH^{d}(\partial\Omega\cap B(x,r)) <\infty\right\}.
\end{equation}
See also \cite{Azz15}, which simplifies some of the technical arguments in \cite{DJ90} and \cite{Bad12}. 

However, in \cite[Theorem 1.2]{AMT15}, the second and third authors along with Tolsa (using a deep result of Wolff \cite{Wolff91}) constructed a two-sided NTA domain $\Omega$ with $\cH^{d}(\d\Omega)<\infty$ but $\omega_{\Omega}\not\ll \cH^{d}|_{\d\Omega}$. See also \cite{A16,LN12} for the p-harmonic version of these results. 



Recently, the second and third author, together with Hofmann, Martell, Mayboroda, Tolsa, and Volberg showed in \cite[Theorem 1.1 (a)]{AHMMMTV} that rectifiability of harmonic measure (rather than rectifiability of the boundary in the classical sense) is in fact {\it necessary}.

\begin{theorem}{\cite[Theorem 1.1(a)]{AHMMMTV}}\label{t:7per}
Let $\Omega\subset\mathbb{R}^{d+1}$ be open and connected and $E\subset\partial\Omega$ with $\cH^{d}(E)<\infty$. If $\omega_{\Omega}\ll \cH^{d}$ on $E$, then $E$ may be covered by countably many Lipschitz graphs up to a set of $\omega_{\Omega}$-measure zero.
\end{theorem}

This is like a higher dimensional version of Pommerenke's theorem, only that now absolute continuity implies rectifiability of harmonic measure (and in fact the existence of a rectifiable set in the boundary of positive $d$-measure) rather than the existence of cone points. The theorem (and also Pommerenke's theorem) are false in higher dimensions without the assumption $\cH^{d}(E)<\infty$: Wolff showed in \cite[Theorem 3]{Wolff91} that there are domains $\Omega\subset \bR^{3}$ for which the harmonic measure of any $2$-dimensional set (like a Lipschitz graph) is zero. 



It is natural to ask about when we alternatively have that $\cH^{d}|_{\d\Omega} \ll \omega_{\Omega}$, and there has been much work on this as well. In \cite[Theorem 1.1(b)]{AHMMMTV}, it was also shown that if $\cH^{d}|_{E}\ll \omega_{\Omega}|_{E}$ for some Borel set $E\subset \d\Omega$, then $E$ is $d$-rectifiable. If $\Omega$ is uniform with rectifiable and lower regular boundary, and $\cH^{d}|_{\d\Omega}$ is a Radon measure, then the third author showed in \cite[Theorem 1.1]{Mou15} that rectifiability implies $\cH^{d}|_{\d\Omega}\ll  \omega_{\Omega}$. Independently and simultaneously in \cite[Theorem 1.2]{ABHM15}, the first author, Badger, Hofmann, and Martell showed that when $\Omega$ is a $1$-sided NTA domain whose boundary $\partial\Omega$ is $d$-Ahlfors regular then $\Omega$ is rectifiable if and only if $\cH^{d}|_{\d\Omega}\ll \omega_{\Omega}$, and in fact it was shown in that this is equivalent to the existence of a few other geometric decompositions of the boundary. Moreover, it was proven that this also held for some more general elliptic measures (Theorem 1.3 in \cite{ABHM15}) rather than just harmonic measure, whereas the techniques in \cite{AHMMMTV}, for example, do not apply to this setting. For the specific class of elliptic measures, see Definition \ref{d:KP} below.
%
%
%
%

The first author, Bortz, Hofmann, and Martell, showed in \cite[Theorem 2.1]{ABHM16} that if $E$ is a closed $d$-rectifiable set satisfying a condition weaker than lower $A$-Ahlfors $d-$regularity condition and having locally finite $\cH^{d}$ measure then any Borel subset of $E$ with positive $\cH^{d}$ measure has non-zero harmonic measure in at least one of the connected components of $\mathbb{R}^{d+1}\setminus E$. This was also shown in {\cite[Theorem 1.4]{Mou16}} but under the assumption that the measure theoretic boundary had full measure in the boundary. Combining their result with \Theorem{7per}, they get the following classification theorem.

\begin{theorem}[{\cite[Theorem 2.9]{ABHM16} }]\label{t:h<w}
Let $\Omega\subset \bR^{d+1}$ be a bounded domain so that $\d\Omega$ has locally finite $\cH^{d}$-measure. Suppose that $\d\Omega$ has the Weak Lower Ahlfors-David regular condition (WLADR), meaning that for $\cH^{d}$-almost every $x\in \d\Omega$, we have 
\begin{equation*}
\limsup_{r\rightarrow 0} \inf\ck{\cH^{d} (\d\Omega\cap B(y,s)): y\in \d\Omega\cap B(x,r),0<s<r}\gtrsim  s^{d}>0.
\end{equation*}
Further, suppose that the interior measure theoretic boundary has full measure, meaning that for $\cH^{d}$-almost every $x\in \d\Omega$ we have 
\[
\frac{\cH^{d+1}(B(x,r)\cap \Omega)}{\cH^{d+1}(B(x,r))}>0.\]
Then $\cH^{d}|_{\d\Omega}\ll \omega_{\Omega}$ if and only if $\d\Omega$ is $d$-rectifiable.
\end{theorem}

See also Theorem A.1 and Theorem A.3 in \cite{ABHM16} for localized version of \Theorem{h<w} and for decomposing $\partial\Omega$ as a rectifiable portion, where surface measure is absolutely continuous with respect to harmonic measure, and a purely $d$-unrectifiable set with
vanishing harmonic measure.

\subsection{Main Results}

Our first main result is a generalization of Wu's theorem for domains that have uniformly large complements rather than exterior corkscrews. 

\begin{definition}\label{d:LBB}
A domain $\Omega\subset \bR^{d+1}$ is said to have {\it big boundary in $K$} for some set $K\subset \bR^{d+1}$ if there is $c_{K}>0$ so that
\begin{equation}\label{e:kcontent}
\cH^{d}_{\infty} ( B\backslash \Omega)\geq c_{K} r_{B}^{d} \mbox{ for all $B$ centered on $K\cap \d\Omega$ with $0<r_{B}<\diam K$}.
\end{equation}
We will say that $\Omega$ has {\it big boundary} if it has big boundary in $\d\Omega$, or in other words,
\begin{equation}\label{e:dcontent}
\cH^{d}_{\infty} ( B\backslash \Omega)\geq c r_{B}^{d} \mbox{ for all $B$ centered on $\d\Omega$ with $r_{B}<\diam \d\Omega$}.
\end{equation}
\end{definition} 

Our theorem also holds more generally for class of elliptic measures $\omega_{\Omega}^{\mathcal{L},X}$   satisfying the following condition taken from \cite{KP01}.

\begin{definition}\label{d:KP}
Let $\delta(X)=\mbox{dist}(X,\partial\Omega)$. We will say that an elliptic operator $\mathcal{L}=-\div \mathcal{A}\grad$ satisfies the {\it Kenig-Pipher condition (or KP-condition)} if $\mathcal{A}=(a_{ij}(X))$ is a uniformly elliptic real matrix that has distributional derivatives such that
\[
\ve_{\Omega}^{\mathcal{L}}(Z):=\sup\{\delta(X)|\grad a_{ij}(X)|^{2}: X\in B(Z,\delta(Z))/2,\;\; 1\leq i,j\leq d+1\}
\]
is a Carleson measure in $\Omega$, by which we mean for all $x\in \d\Omega$ and $r\in (0,\diam \d\Omega)$,
\[
\int_{B(x,r)\cap \Omega}\ve_{\Omega}^{\mathcal{L}}(Z)dZ\leq Cr^{d}.\]
\end{definition}

\def\CorollaryI{Corollary \hyperref[c:cori]{I} }
\def\TheoremI{Theorem \hyperref[t:thmi]{I} }
\def\TheoremIV{Theorem \hyperref[t:thmi]{IV} }
\def\TheoremII{Theorem \hyperref[t:thmi]{II} }
\def\TheoremIII{Theorem \hyperref[t:thmi]{III} }
\def\assumption{Let $\Omega\subset \bR^{d+1}$ be a regular domain with big boundary in some ball $B_{0}$ centered on $\d\Omega$. }
\begin{thmi}\label{t:thmi}{\it 
\assumption Let $\mathcal{L}$ be an elliptic operator satisfying the \hyperref[d:KP]{KP-condition}. If $d=1$ and $\Omega$ is unbounded, assume either that $\infty$ is regular for $\Omega$ or $\omega_{\Omega}^{\mathcal{L}}(\infty)=0$. Suppose $\Gamma\subset \bR^{d+1}$ is $A$-Ahlfors $d-$regular and splits $\bR^{d+1}$ into two NTA domains $\Omega_{1}$ and $\Omega_{2}$. If $E\subset \d\Omega\cap \Gamma\cap B_{0}$ is a Borel set, then
\begin{equation}\label{e:converse}
 \cH^{d}(E)=0 \mbox{ implies } \omega_{\Omega}^{\mathcal{L},X_{0}}(E)=0.
 \end{equation}}
\end{thmi}

The result does not hold without the \hyperref[d:KP]{KP-condition}, even in the case that $\Omega$ is a half space and $\Gamma=\d\Omega$  \cite{CFK81,Swe92,Wu94}. Even in the half plane setting, some sort of Dini or Carleson condition on the coefficients is typically required, see \cite{FJK84,FKP91,KP01} and the references therein.

In the case $\mathcal{L}=\Delta$, if $d=1$, then our assumptions imply $\cH^{1}(\d\Omega)>0$, so that $\d\Omega$ is nonpolar \cite[Theorem 11.14, p. 207]{HKM}. Domains with nonpolar boundaries are Greenian by Myrberg's Theorem   \cite[Theorem 5.3.8, p. 133]{AG} and harmonic measures for unbounded Greenian domains give zero measure to $\infty$ \cite[Example 6.5.6, p. 179]{AG}. Thus, we have the following corollary for the case of harmonic measure.

\begin{cori}\label{c:cori}
\assumption Suppose $\Gamma\subset \bR^{d+1}$ is $A$-Ahlfors $d-$regular and splits $\bR^{d+1}$ into two NTA domains $\Omega_{1}$ and $\Omega_{2}$. If $E\subset \d\Omega\cap \Gamma\cap B_{0}$ is a Borel set, then $\omega_{\Omega}\ll \cH^{d}$ on $\d\Omega\cap \Gamma\cap B_{0}$.
\end{cori}

This corollary is, to our knowledge, also new in the plane, as we have no topological assumptions on $\Omega$ like simple connectedness. This is particularly interesting in light of \Theorem{bj}; while $\omega_{E^{c}}(\Gamma)>0$ for some Ahlfors regular curve, by \TheoremI we must have $\omega_{E^{c}}(\Gamma)=0$ whenever $\Gamma$ is a bi-Lipschitz curve.

The big boundary condition cannot be loosened too much, as one cannot change the $d$ to some $s<d$ in \eqn{kcontent}. Just consider traditional harmonic measure and take any fractal set $E$ in $\bR^{d}$ satisfying $\cH^{s}_{\infty} (E\cap B)\geq c r_{B}^{s}$ for all $B$ centered on $E$ with $r_{B}<\diam E$, and then consider $\Omega=\bR^{d+1}\backslash E$. Then \TheoremI fails with $\Gamma=\bR^{d}$. The Ahlfors regularity assumption on $\Gamma$ cannot be relaxed either, by the counterexample in \cite{AMT15} mentioned earlier just below \eqn{bad}.

Our second main result shows that rectifiability of harmonic measure impies the existence of cone points. Recall that a point $x\in \d\Omega$ is a {\it cone point} for $\Omega$  if there is a vector $v\in \bS^{d}$, $r>0$, and $\alpha>0$ so that 
\[
C(x,v,\alpha,r):= \{y\in B(x,r): (y-x)\cdot v> \alpha |y-x|\}\subset \Omega.\]
A set $\Gamma$ is a {\it Lipschitz graph} if it is a rotation and translation of a set of the form $\{(x,f(x)):x\in \bR^{d}\}$ where $f:\bR^{d}\rightarrow \bR$ is Lipschitz.

\begin{thmii}\label{t:thmii}{\it
\assumption Let $\omega_{\Omega}$ be its harmonic measure and let $\Gamma$ be a Lipschitz graph. Then $\omega_{\Omega}$-almost every point in $\Gamma\cap \d\Omega\cap B_{0}$ is a cone point for $\Omega$.}
\end{thmii}

By combining \CorollaryI and \TheoremII with \Theorem{7per}, we obtain the following generalization of the results of McMillan and Pommerenke and completely characterize sets of absolute continuity with finite $\cH^{d}$-measure both in terms of the cone point condition and in terms of the rectifiable structure where $\Omega$ has big boundary.

\begin{thmiii}\label{t:thmiii}{\it
\assumption Let $E\subset \d\Omega\cap B_{0}$ be a Borel set with $\cH^{d}(E)<\infty$. Then the following statements are equivalent:
\begin{enumerate}
\item $\omega_{\Omega}|_{E}\ll \cH^{d}|_{E}$.
\item $E$ may be covered up to $\omega_{\Omega}$-measure zero by countably many Lipschitz graphs.
\item  $\omega_{\Omega}$-almost every point in $E$ is a cone point for $\Omega$.
\end{enumerate}
Moreover, if $F$ is the set of cone points in $\Omega\cap B_{0}$, then 
\begin{equation}\label{e:whw}
\omega_{\Omega}|_{F}\ll \cH^{d}|_{F}\ll \omega_{\Omega}|_{F}.
\end{equation}}
\end{thmiii}

Note that the condition $\cH^{d}(E)<\infty$ is crucial, and so we do not recover Pommerenke's theorem in the plane. However, the above version has the advantage of holding in all dimensions and for sets that are not simply connected.\\

%


\subsection{Outline}

In Section \ref{s:prelim}, we recall first some basic notation, the sawtooth construction of NTA domains due to Hofmann and Martell \cite{HM14}, and some preliminarly lemmas about harmonic and elliptic measures that will be used often. The reader unfamiliar with elliptic measures can assume all measures in this paper are harmonic. In Section \ref{s:I}, we prove the main lemma of the paper, which states in some sense that if we look at the harmonic measure of a set $E\subset \d\Omega$ inside the boundary of two NTA domains, then harmonic measure with respect to one of those NTA domains must be large. We then use that to prove \TheoremI. In Section \ref{s:cones}, we use this lemma and introduce some background on the tangent measures of Preiss \cite{Pr87} in order to prove \TheoremII. In Section \ref{sec:III}, we use the previous two theorems along with \Theorem{7per} to give the characterization \TheoremIII.\\
\\
1.4 {\bf Acknowledgments.} The authors are grateful to Raanan Schul and Xavier Tolsa for their helpful discussions and encouragement, as well as Jos\'{e} Mar\'{i}a Martell for pushing us to derive the local version of our result. The third author would like to thank D. Betsakos for answering several questions concerning the strong Markov property of Brownian motion. Research for this article was carried out while the first author M. Akman was visiting the mathematics department at the Universitat Aut\`onoma de Barcelona, the author would like to thank to  the department for its hospitality. The first author was supported by ICMAT Severo Ochoa project SEV-2015-0554, and also acknowledges that the research leading to these results has received funding from the European Research Council under the
European Union's Seventh Framework Programme (FP7/2007-2013)/ ERC agreement no. 615112
HAPDEGMT. J. Azzam  and M. Mourgoglou were supported by the ERC grant 320501 of the European Research Council (FP7/2007-2013).

\section{Preliminaries}
\label{s:prelim}

We will write $a\lesssim b $ if there is a constant $C>0$ so that $a\leq C b$ and $a\lesssim_{t} b$ if the constant depends on the parameter $t$. As usual we write $a\sim b$ and $a\sim_{t} b$ to mean $a\lesssim b \lesssim a$ and 
$a\lesssim_{t} b \lesssim_{t} a$ respectively. We will assume all implied constants depend on $d$ and hence write $\sim$ instead of $\sim_{d}$.

Whenever $A,B\subset\mathbb{R}^{d+1}$ we define
\[
\mbox{dist}(A,B)=\inf\{|x-y|;\, x\in A, \, y\in B\}, \, \mbox{and}\, \, \mbox{dist}(x,A)=\mbox{dist}(\{x\}, A). 
\]
Let $\diam A$ denote the diameter of $A$ defined as
\[
\diam A=\sup\{|x-y|;\, x,y\in A\}.
\]
Whenever $A\subset\mathbb{R}^{d+1}$ and $0<\delta\leq\infty$ we define $(d,\delta)-$\textit{Hausdorff content} of $A$, denoted by $\cH^{d}_{\delta}(A)$, as
\[
\cH^{d}_{\delta}(A)=\inf\left\{\sum(\diam A_i)^{d};\, A\subset\bigcup\limits_{i} A_i, \, (\diam A_i)\leq\delta\right\}.
\]
The $d$-\textit{dimensional Hausdorff measure} of $A$, denoted as $\cH^{d}(A)$, defined as
\[
\cH^{d}(A)=\lim\limits_{\delta\downarrow 0} \cH^{d}_{\delta}(A),
\]
and $\cH^{d}_{\infty}(A)$ is called the {\it $d$-dimensional Hausdorff content} of $A$. 

For a Euclidean ball $B$, we will denote its radius by $r_{B}$.
\subsection{NTA domains and sawtooth regions}
\label{ss:defs} 

\begin{definition}[\bf Ahlfors regular]
\label{d:ADR}
 We say that a closed set $E \subset \mathbb{R}^{d+1}$ is {\it $A$-Ahlfors $d$-regular} if
there is some uniform constant $A$ such that
\begin{align*}
\frac1A\, r^d \leq \cH^{d}(E\cap B(x,r)) \leq A\, r^d \quad \forall\, r\in(0,\diam(E)),\, x \in E.
\end{align*}
\end{definition}

Note that if $E$ is $A$-Ahlfors $d-$regular, then for any $F\subset E$,
\begin{equation}\label{e:hhfin}
\cH^{d}(F)\sim_{A} \cH^{d}_{\infty}(F).\end{equation}
Firstly, we have by definition that $\cH^{d}_{\infty}(F)\leq \cH^{d}(F)$. Conversely, if $A_{i}$ is any cover of $F$, then 
\[
\cH^{d}(F)\leq \sum \cH^{d}(A_{i}\cap E)\leq A\sum (\diam A_{i})^{d}\]
and infimizing over all such covers gives $\cH^{d}(F)\leq A\cH^{d}_{\infty}(F)$, which proves \eqn{hhfin}.

Following \cite{JK82}, we state the definition of Corkscrew condition, Harnack Chain condition, and NTA domains.\\

\begin{definition}[\bf Corkscrew condition]
\label{d:CS}
We say that an open set $\Omega\subset \mathbb{R}^{d+1}$
satisfies the interior {\it $c$-Corkscrew condition} if for some uniform constant $c$, $0<c<1$, and
for every ball $B$ centered on $\partial\Omega$ with
$0<r_{B}<\diam(\partial\Omega)$, there is a ball
$B(X_B,cr_{B})\subset B\cap\Omega$.  The point $X_B\subset \Omega$ is called
a {\it corkscrew point relative to} $B,$ (or, relative to $B$). We note that  we may allow
$r_B<C\diam(\partial\Omega)$ for any fixed $C$, simply by adjusting the constant $c$. If $\Delta=\d\Omega\cap B$ is the corresponding surface ball, we will write $X_{\Delta}=X_{B}$.
\end{definition}
\begin{definition}[{\bf Exterior
Corkscrew condition}]
\label{d:dyadcork}
We say that an open set $\Omega\subset \mathbb{R}^{d+1}$
satisfies the exterior {\it $c$-Corkscrew condition} if for some uniform constant $c$, $0<c<1$, and
for every ball $B$ centered on $\partial\Omega$ with
$0<r_{B}<\diam(\partial\Omega)$,
there is a ball of radius $cr_{B}$ contained in $B\backslash \overline{\Omega}$. 
\end{definition}

\begin{definition}[\bf Harnack Chain condition]
We say that $\Omega$ satisfies the $C$-{\it Harnack Chain condition} if there is a uniform constant $C$ such that
for every $\rho >0,\, \Lambda\geq 1$, and every pair of points
$X,X' \in \Omega$ with $\delta(X),\,\delta(X') \geq\rho$ and $|X-X'|<\Lambda\,\rho$, there is a chain of
open balls
$B_1,\dots,B_N \subset \Omega$, $N\leq C(\Lambda)$,
with $X\in B_1,\, X'\in B_N,$ $B_k\cap B_{k+1}\neq \emptyset$, $C^{-1}\diam (B_k) \leq \dist (B_k,\partial\Omega)\leq C\diam (B_k)$, and $\diam B_{k}\cap B_{k+1}\geq C^{-1}\max\{r_{k},r_{k+1}\}$. Such a sequence is called a {\it Harnack Chain}.
\end{definition}

\begin{definition}[\bf 1-sided NTA domain]
\label{d:1nta}
If $\Omega$ satisfies both the  $C$-Harnack Chain and the $C^{-1}$-Corkscrew conditions, then we say that
$\Omega$ is a {\it 1-sided $C$-NTA domain}.
\end{definition}
\begin{definition}[\bf NTA domain]
\label{d:nta}
We say that a domain $\Omega$ is a {\it $C$-NTA  domain} if
it is a 1-sided $C$-NTA domain and satisfies the $C^{-1}$-exterior corkscrew condition.
\end{definition}

\subsection{Dyadic grids and sawtooths}
\label{ss:grid}

In this subsection, we follow \cite{ABHM15,HM14} and introduce \textit{dyadic grids, sawtooth domains,} and the {\it Carleson box}. We begin by giving a lemma concerning the existence of dyadic grids, which can be found in \cite{DS,of-and-on,Christ-T(b)}.

\begin{lemma}[\bf Existence and properties of the ``dyadic grid'']
\label{l:Christ}
If $E\subset \mathbb{R}^{d+1}$ is $A$-Ahlfors $d$-regular, then there exist
constants $ a_0>0$, $\eta>0$, and $C_1<\infty$, depending only on $A$ and $d$, and for each $k \in \mathbb{Z}$
there exists a collection of open sets (which we will call``cubes'') 
$$
\mathbb{D}_k:=\{Q_{j}^k\subset E: j\in \mathfrak{I}_k\},$$ that are countable unions of relatively open balls in $E$, where
$\mathfrak{I}_k$ denotes some (possibly finite) index set depending on $k$, satisfying the following properties.

\begin{list}{$(\theenumi)$}{\usecounter{enumi}\leftmargin=.8cm
\labelwidth=.8cm\itemsep=0.2cm\topsep=.1cm
\renewcommand{\theenumi}{\roman{enumi}}}

\item $\cH^{d}\ps{E\backslash \bigcup_{j}Q_{j}^k}=0\,\,$ for each
$k\in{\mathbb Z}$.

\item If $m\geq k$ then either $Q_{i}^{m}\subset Q_{j}^{k}$ or
$Q_{i}^{m}\cap Q_{j}^{k}=\emptyset$.

\item For each $(j,k)$ and each $m<k$, there is a unique
$i$ such that $Q_{j}^k\subset Q_{i}^m$.

\item The diameter of each $Q_{j}^k$ is at most $C_12^{-k}$.

\item Each $Q_{j}^k$ contains some surface ball $\Delta \big(x^k_{j},a_02^{-k}\big):=
B\big(x^k_{j},a_02^{-k}\big)\cap E$.

\item $\cH^{d}\left(\left\{x\in Q^k_j:{\rm dist}(x,E\setminus Q^k_j)\leq \tau \,2^{-k}\right\}\right)\leq
C_1\,\tau^\eta\,\cH^{d}\left(Q^k_j\right)$ for all $k$ and $j$ and for all $\tau\in (0,a_0)$.
\end{list}
\end{lemma}

Some notations and remarks are in order concerning this lemma.

\begin{list}{$\bullet$}{\leftmargin=0.4cm  \itemsep=0.2cm}

\item In the setting of a general space of homogeneous type, this lemma has been proved by Christ
\cite{Christ-T(b)}, with the
dyadic parameter $1/2$ replaced by some constant $\delta \in (0,1)$.
In fact, one may always take $\delta = 1/2$ (cf.  \cite[Proof of Proposition 2.12]{HMMM14}).
In the presence of Ahlfors regular property, the result already appears in \cite{DS,of-and-on}. For geometrically doubling metric spaces, an improved version of these cubes were developed by Martikainen and Hyt\"onen \cite{HM12}.

\item  For our purposes, we may ignore those
$k\in \mathbb{Z}$ such that $2^{-k} \gtrsim {\rm diam}(E)$ whenever $E$ is bounded.

\item  We shall denote by  $\mathbb{D}=\mathbb{D}(E)$ the collection of all relevant
$Q^k_j$. That is, $$\mathbb{D} := \bigcup\limits_{k} \mathbb{D}_k,$$
where the union runs only
over those $k$ such that $2^{-k} \lesssim  {\rm diam}(E)$ whenever $E$ is bounded.

%

\item For a dyadic cube $Q\in \mathbb{D}_k$, we
set $\ell(Q) = 2^{-k}$ and we call this quantity the ``side length''
of $Q$.  Evidently, $\ell(Q)\sim \diam(Q).$
%

\item Properties $(iv)$ and $(v)$ imply that for each cube $Q\in\mathbb{D}_k$,
there exists a point $x_Q\in E$, a Euclidean ball $B(x_Q,r_Q)$ and corresponding surface ball
$\Delta(x_Q,r_Q):= B(x_Q,r_Q)\cap E$ such that
\begin{align*}
& c\ell(Q)\leq r_Q\leq \ell(Q),\\
& \Delta(x_Q,2r_Q)\subset Q \subset \Delta(x_Q,Cr_Q),
\end{align*}
for some uniform constants $c$ and  $C$, and 
\begin{equation}\label{e:bqbrempty}
B(x_{Q},r_{Q})\cap B(x_{R},r_{R})\neq\emptyset \;\; \mbox{ if and only if }\;\; Q=R.
\end{equation} 
We shall denote this ball and surface ball by $B_Q:= B(x_Q,r_Q)$ and $\Delta_Q:= \Delta(x_Q,r_Q)$, respectively,
and we shall refer to the point $x_Q$ as the ``center'' of $Q$.


\end{list}

It will be useful  to dyadicize the Corkscrew condition and to specify
precise Corkscrew constants.
Let us now specialize to the case that  $E=\partial\Omega$ is $d$-Ahlfors regular
with $\Omega$ satisfying the Corkscrew condition.
Given $Q\in \mathbb{D}(\partial\Omega)$, we
shall sometimes refer to a corkscrew point $X_Q$ relative to $Q$, which define to be a corkscrew point $X_\Delta$ relative to the ball
$B_{Q}$ 
We note that $\delta(X_Q) \sim \dist(X_Q,Q) \sim \diam(Q)$.

Following \cite[Section 3]{HM14} we next introduce the notion of \textit{\bf Carleson region} and \textit{\bf discretized sawtooth}.  Given a cube $Q\in\mathbb{D}(\partial\Omega)$, the \textit{\bf discretized Carleson region $\mathbb{D}_{Q}$} relative to $Q$ is defined by
\[
\mathbb{D}_{Q}=\{Q'\in\mathbb{D}(\partial\Omega):\, \, Q'\subset Q\}.
\]
Let $\F$ be family of disjoint cubes $\{Q_{j}\}\subset\mathbb{D}(\partial\Omega)$. The \textit{\bf global discretized sawtooth region} relative to $\F$ is the collection of cubes $Q\in\mathbb{D}$  that are not contained in any $Q_{j}\in\F$;
\[
\mathbb{D}_{\F}:=\mathbb{D}\setminus \bigcup\limits_{Q_{j}\in\F}\mathbb{D}_{Q_{j}}.
\]
For a given $Q\in\mathbb{D}$ the {\bf local discretized sawtooth region} relative to $\F$ is the collection of cubes in $\mathbb{D}_{Q}$ that are not in contained in any $Q_{j}\in\F$;
\[
\mathbb{D}_{\F,Q}:=\mathbb{D}_{Q}\setminus \bigcup\limits_{Q_{j}\in \F} \mathbb{D}_{Q_{j}}=\mathbb{D}_{\F}\cap \mathbb{D}_{Q}.
\]
We also introduce the ``geometric'' Carleson regions and sawtooths. In the sequel, $\Omega \subset \mathbb{R}^{d+1}$ ($d\geq 2$) will be a 1-sided NTA domain with ADR boundary. Let $\mathcal{W}=\W(\Omega)$ denote a collection
of (closed) dyadic Whitney cubes of $\Omega$, so that the cubes in $\mathcal{W}$
form a covering of $\Omega$ with non-overlapping interiors, and  which satisfy

\begin{align*}
\begin{split}
& 4\, {\rm{diam}}\,(I)\leq \dist(4 I,\partial\Omega) \leq  \dist(I,\partial\Omega) \leq 40 \, {\rm{diam}}\,(I),\\
& \diam(I_1)\sim \diam(I_2), \mbox{ whenever $I_1$ and $I_2$ touch.}
\end{split}
\end{align*}
Let $X(I)$ denote the center of $I$, let $\ell(I)$ denote the side length of $I$,
and write $k=k_I$ if $\ell(I) = 2^{-k}$.

Given $0<\lambda<1$ and $I\in\W$ we write $I^*=(1+\lambda)I$ for the ``fattening'' of $I$. By taking $\lambda$ small enough,  we can arrange matters so that, first, $\dist(I^*,J^*) \sim \dist(I,J)$ for every
$I,J\in\W$, and secondly, $I^*$ meets $J^*$ if and only if $\partial I$ meets $\partial J$.
 (Fattening ensures $I^*$ and $J^*$ overlap for
any pair $I,J \in\W$ whose boundaries touch. Thus, the Harnack Chain property holds locally in $I^*\cup J^*$ with constants depending on $\lambda$.)  By picking $\lambda$ sufficiently small, say $0<\lambda<\lambda_0$, we may also suppose that there is $\tau\in(1/2,1)$ such that for distinct $I,J\in\W$, $\tau J\cap I^* = \emptyset$. In what follows we will need to work with dilations $I^{**}=(1+2\,\lambda)I$ and in order to ensure that the same properties hold we further assume that $0<\lambda<\lambda_0/2$.

For every $Q$ we can construct a family $\W_Q^*\subset \W$ and define
\begin{align}
\label{e:whitney3}
U_Q := \bigcup_{I\in\,\mathcal{W}^*_Q} \mbox{int }I^*\,,
\end{align}
where $\int A=A^{\circ}$ denotes the interior of $A$, satisfying the following properties:
$X_Q\in U_Q$ and there are uniform constants $k^*$ and $K_0$ such that
\begin{align}
\label{e:whitney2}
\begin{split}
& k(Q)-k^*\leq k_I \leq k(Q) + k^*\, \quad  \forall\, I\in \mathcal{W}^*_Q,\\
& X(I) \rightarrow_{U_Q} X_Q\,\quad \forall\, I\in \mathcal{W}^*_Q,\\
&\dist(I,Q)\leq K_0\,2^{-k(Q)}=K_{0}\ell(Q) \, \quad \forall\, I\in \mathcal{W}^*_Q\,.
\end{split}
\end{align}
Here $X(I) \rightarrow_{U_Q} X_Q$ means that the interior of $U_Q$ contains all the balls in
a Harnack Chain (in $\Omega$) connecting $X(I)$ to $X_Q$.
The constants  $k^*$, $K_0$ and the implicit constants in the condition $X(I)\to_{U_Q} X_Q$ in \eqn{whitney2}
depend on at most allowable
parameters and on $\lambda$. The reader is referred to \cite{HM14} for full details.

We also recall from \cite[Equation (3.48)]{HM14} that 
\begin{equation}\label{e:XQ}
X_{Q}\in U_{Q} \mbox{ and } X_{R}\in U_{Q} \mbox{ for each child $R$ of $Q$}.
\end{equation}

For a given $Q\in\mathbb{D}$, the {\bf Carleson box} relative to $Q$ is defined by
\[
T_{Q}:=\bigcup\limits_{Q'\in\mathbb{D}_{Q}} U_{Q'}.
\]
For a given family $\F$ of disjoint cubes $\{Q_{j}\}\subset\mathbb{D}$, the {\bf global sawtooth region} relative to $\F$ is
\[
\Omega_{\F}:=\bigcup\limits_{Q'\in\mathbb{D}_{\F}} U_{Q'}.
\]
Finally, for a given $Q\in\mathbb{D}$ we define the {\bf local sawtooth region} relative to $\F$ by
\[
\Omega_{\F,Q}:=\bigcup\limits_{Q'\in\mathbb{D}_{\F,Q}}U_{Q'}.
\]
For later use we recall \cite[Proposition 6.1]{HM14}:
 \begin{align}
\label{p:HM1}
Q\setminus \bigg( \bigcup_{Q_{j}\in \F} Q_{j}\bigg) \subset \partial\Omega \cap\partial\Omega_{\F,Q}
\subset
\overline{Q}\setminus \bigg( \bigcup_{Q_{j}\in \F} Q_{j}^{\circ}\bigg).
 \end{align}

\begin{lemma}[{\cite[Lemma 3.55]{HM14}}] \label{l:HM3}
 Suppose that $\Omega$ is a 1-sided NTA domain with an ADR boundary. Given $Q\in \bD$, there is a ball $B_{Q}'  \subset B_{Q}$, with $r_{B_{Q}'}\sim l(Q) \sim r_{B_{Q}}$, such that
\begin{equation*}
 B'_{Q} \cap \Omega \subset T_Q 
 \end{equation*}
and such that for every pairwise disjoint family $\F \subset \bD$, and for each $Q_{0} \in \bD$ containing $Q$, we have
 \begin{equation*}
 B_{Q}' \cap \Omega_{\F,Q_0} = B_{Q}' \cap \Omega_{\F,Q}.
 \end{equation*}
\end{lemma}

\begin{lemma}[{\cite[Lemma 3.61]{HM14}}]
\label{l:HM1}
Suppose that $\Omega$ is a $1-$sided NTA domain with $d$-Ahflors regular boundary. Then all of its Carleson boxes $T_{Q}$ and $T_{\Delta}$, and the sawtooth regions $\Omega_{\mathcal{F}}$ and $\Omega_{\mathcal{F},Q}$, are also $1-$sided NTA domains with $d$-Ahflors regular boundary. 
\end{lemma}
\begin{lemma}[{\cite[Lemma 3.62]{HM14}}]
\label{l:HM2}
Suppose that $\Omega$ is a $1-$sided NTA domain with $d$-Ahflors regular boundary. Assume also that $\Omega$ satisfies the exterior Corkscrew condition. Then all of its Carleson boxes $T_{Q}$ and $T_{\Delta}$, and sawtooth regions $\Omega_{\mathcal{F}}$ and $\Omega_{\mathcal{F},Q}$ satisfy the exterior Corkscrew condition.
\end{lemma}

The original statement spoke of the qualitative exterior corkscrew condition rather than the full corkscrew condition, but of course having the exterior corkscrew condition is stronger and the proofs of these result are identical in this case.

\begin{remark}
\label{r:fatstar}
We also define $T_{Q}^{*},\Omega_{\F}^{*}$, and $\Omega_{\F,Q}^{*}$ the same way but with $U_{Q}^{*}$ in place of $U_{Q}$, where
\[U_{Q}^{*} := \bigcup_{I\in\,\mathcal{W}^*_Q} \mbox{int } I^{**}.\]

Then the statements and lemmas above are also true for $T_{Q}^{*},\Omega_{\F}^{*}, \Omega_{\F,Q}^{*},$ and $U_{Q}^{*}$.
\end{remark}
%
\subsection{Elliptic and harmonic measures}

In this section we assume that $\Omega \subset \R^{d+1}$. If $\Omega$ is unbounded, we denote the extended boundary of $\Omega$ by $\d_\infty \Omega=\d\Omega\cup \{\infty\}$; otherwise, we set $\d_{\infty}\Omega= \d\Omega$. 

From now on, $\mathcal{A}=(a_{ij}(X))_{1\leq j \leq d+1}$ will always be a {\it uniformly elliptic} real matrix in $\Omega$, meaning there is $\lambda>0$ so that
\[
 \mathcal{A}(X) \xi \cdot \xi \geq \lambda |\xi |^2 \;\; \mbox{ for all $\xi \in \R^{d+1}$ and a.e. $X \in \Omega$}\]
with $a_{ij} \in L^\infty(\Omega; \R)$. We define the second order elliptic operator $\mathcal{L} = -\div \mathcal{A} \nabla$ and we will say that a function $u \in W^{1,2}_{loc}(\om)$ is a {\it solution} of the equation $\mathcal{L} u=0$ in $\Omega$ if 

\begin{equation*}
\int \mathcal{A} \nabla u \nabla \Phi=0 \;\; \mbox{ 
for all $\Phi \in C^\infty_0(\om)$}.
\end{equation*} 

We also say that  $u \in W^{1,2}_{loc}(\om)$ is a {\it supersolution} (or subsolution) for $\mathcal L$ in  $\Omega$  if $\int \mathcal{A} \nabla u \nabla \Phi \geq 0$   (or $\int \mathcal{A} \nabla u \nabla \Phi \leq 0$) for all non-negative $\Phi \in C^\infty_0(\om)$.

We next introduce \textit{upper (or lower) Perron solutions} by following \cite[Section 9]{HKM}. To this end, let $f:\d_\infty \Omega \to [-\infty, \infty]$ be a function. The upper class $\mathcal U_f$ (or lower class $\mathcal L_{f}$) of $f$ consists of all functions $u$ such that
\begin{itemize}
\item[(i)] $u$ is a supersolution (or subsolution) for $\mathcal L$ in $\Omega$, 
\item[(ii)] $u$ is bounded below (or above), and
\item[(iii)] $\liminf_{x \to y} u(x) \geq f(y)$ (or  $\limsup_{x \to y} u(x) \leq f(y)$), for all $y \in \d_\infty \Omega$.
\end{itemize} 
The function $\overline{H}_f= \inf\{u: u \in \mathcal{U}_f\}$ is the {\it upper Perron solution} of $f$ in $\Omega$ for the elliptic operator $\mathcal L$ and $\underline{H}_f= \sup\{u: u \in \mathcal{L}_f\}$ is the {\it lower Perron solution}.  If $\mathcal{U}_f=\emptyset$ then we set $\overline{H}_f=\infty$.

If $E \subset \d \Omega$, we define the {\it $\mathcal{L}$-elliptic measure} of $E$ in $\Omega$ with pole at $X\in \Omega$ by
$$\hm(E, \Omega ; \mathcal L)(X) = \overline{H}_{\one_E}(X).$$

We say that a point $x \in \partial_\infty \Omega$ is {\it $\mathcal L$-regular} or just {\it regular} if
$$\lim_{X \to x} \overline{H}_f(X) = f(x),$$
 for every $f \in C(\partial_\infty \Omega)$. 
 Note that, by Wiener's criterion, $x \in \partial \Omega$ is regular if and only if $$\int_0^1 \frac{\textup{cap}(B(x,r) \cap \Omega^c, B(x,2r))}{\textup{cap}(B(x,r),B(x,2r))}\, \frac{dr}{r}=+\infty,$$
 where cap$(\cdot, \cdot)$ stands for the variational $2$--capacity of the condenser $(\cdot, \cdot)$ (see \cite[p. 27]{HKM} for the definition). Note also that by \cite[Lemma 2.14]{HKM}, 
 \[\textup{cap}(B(x,r),B(x,2r)) \approx r^{d-1}.\] 
 Therefore, a point $x \in \partial \Omega$ is  $\mathcal L$-regular if and only if  it is Wiener regular (that is, in the sense of the Laplace operator). Note that if $d \geq 2$ then $\infty$ is always a regular point, while this is not necessarily the case in $\R^2$ (see e.g. \cite[Theorem 6.4.2]{Hel}).

\begin{definition}
A domain $\Omega \subset \R^{d+1}$ is called {\it regular} if every point of $\d_\infty\Omega$ is regular (i.e., if the classical Dirichlet problem is solvable in $\Omega$ for the elliptic operator $\mathcal{L}$). For $K\subset \d\Omega$, we say that $\Omega$ has the {\it capacity density condition (CDC) in $K$} if $ \textup{cap}({B}(x,r) \cap \Omega^c, B(x,2r)) \gtrsim r^{d-1}$, for every $x \in K$ and $r<\diam K$, and that $\Omega$ has the {\it capacity density condition} if it has the CDC in $K=\d\Omega$. 
\end{definition}

\begin{remark}
By Wiener's criterion, it is clear that domains satisfying the CDC are regular for $d \geq 2$.
\end{remark}


Let $\Omega \subset \R^{d+1}$ be a regular domain. If $f \in C(\partial_\infty \Omega)$, then the map $f \mapsto \overline{H}_f$ is a bounded linear functional on $C(\partial_\infty \Omega)$. Therefore, by Riesz representation theorem, there exists a positive measure $\hm_{\Omega}^{\mathcal{L},X}$ (associated to $\mathcal{L}$ and a point $X \in \Omega$) defined on Borel subsets of $\d_\infty \Omega$ so that
$$ \overline{H}_f(X) =\int_{\partial \Omega} f \, d\hm_{\Omega}^{\mathcal{L},X} \;\; \mbox{for all $X \in \Omega$.}$$ 
 It follows from \cite[Theorem 11.1]{HKM} that $\hm_\Omega^{\mathcal{L}, X}(E)=\hm(E, \Omega ; \mathcal L)(X)$. Moreover,  $\omega^{\mathcal{L}, X}(\d_\infty \Omega)=\overline H_1(X) =1$.

\begin{lemma}
If $\Omega\subset \bR^{d+1}$ satisfies \eqn{dcontent}, then it satisfies the CDC.
\end{lemma}

\begin{proof}
This is well known, but we review the details for completeness. Assume first that $h:[0, \infty) \to [0, \infty)$ is a {\it measure function}, i.e.,  a continuous and strictly increasing function such that $h(0)=0$, $\lim_{r \to \infty} h(r)= \infty$ and define 
$$\cH_h(E)= \inf \ck{ \sum_i h(r_i): E \subset \bigcup_i \overline B(x_i, r_i)}.$$
We also denote $\textup{cap}_p$ to be the ordinary variational $p$-capacity of a condenser.

We recall a theorem from \cite{Martio}.

\begin{theorem}[{\cite[Theorem 3.1]{Martio}}]\label{thm:Martio}
Suppose that $K$ is a closed set in $\R^{d+1}$ and $x \in \R^{d+1}$. If $p \in (1, d+1]$ and $h$ is a measure function such that
\begin{equation}\label{e:h-condition}
\int_0^{2r} h(t)^{1/p} t^{-(d+1)/p} \leq A r^{(p-d-1)/p} h(r)^{1/p},
\end{equation} 
for some $A>0$ and for every $r \in(0, r_0]$, then there exists a constant $C>0$ depending on $n, p$ and  $A$, so that
\begin{equation*}
\frac{\cH_h(K \cap \overline{B}(x,r) ) }{h(r)} \leq C\, \frac{ \textup{cap}_p ( K \cap \overline{B} (x,r), B(x,2r) ) }{r^{d+1-p}},
\end{equation*}
for all $r \in (0, r_0]$.
\end{theorem}

If $h(r)=r^d$ and $p=2$  then it is trivial to show that \eqref{e:h-condition} holds for every $r \in (0, \infty)$. Therefore, if we apply Theorem \ref{thm:Martio} for $K= \Omega^c$ and $p=2$, we deduce that \eqn{dcontent} implies the capacity density condition.
\end{proof}

\begin{lemma}[{\cite[Lemma 11.21]{HKM}}]\label{l:bourgain}
Let $\Omega\subset \bR^{d+1}$ be any domain satisfying the CDC condition,  $x_{0}\in \d\Omega$, and $r>0$ so that $\Omega\backslash B(x_{0},2r)\neq\emptyset$. Then 
\begin{equation}\label{e:bourgain}
\hm_{\Omega}^{\mathcal{L},X}(B(x_{0},2r)\geq c >0 \;\; \mbox{ for all }X\in \Omega\cap B(x_{0},r)
\end{equation}
where $c$ depends on $d$ and the constant in the CDC.
\end{lemma}

The above is actually a corollary of \cite[Lemma 11.21]{HKM}, and we refer the reader there to the complete statement.

\begin{lemma}[Harnack's inequality, {\cite[Theorem 6.2]{HKM}}] Let $\Omega\subset \bR^{d+1}$ and let $u$ be a solution for $\mathcal{L}$. There is $c=c(\lambda,d)>0$ so that if $2B\subset \Omega$, then 
\[
\sup_{B} u\leq c \inf_{B} u.\]
\end{lemma}

Thus, if $B_{1},...,B_{N}$ is a Harnack chain of length $N$ composed of corkscrew balls for a domain $\Omega$, then
\[
\sup_{B_{N}}u\lec_{N,\lambda,d} \inf_{B_{1}}u.\]

 \vvs
\begin{lemma}[Carleman's principle, {\cite[Theorem 11.3]{HKM}}]
If $E\subset \d\Omega_{1}\cap \d\Omega_{2}$ and $\Omega_{1}\subset \Omega_{2}$, then $\hm_{\Omega_{1}}^{\mathcal{L},X}(E)\leq \hm_{\Omega_{2}}^{\mathcal{L},X}(E)$. 
\label{l:CP}
\end{lemma}

\begin{lemma}[{Strong maximum principle, \cite[Theorem 6.5]{HKM}}]
A nonconstant solution for $\mathcal L$ in $\Omega$ cannot attain its supremum or infimum in $\Omega$.
\end{lemma}

\begin{lemma}[{\cite[Corollary 11.10]{HKM}}]\label{l:MP}
Suppose that $u$ is bounded above and a subsolution for $\mathcal L$ in $\Omega$. If $\omega_\Omega^{\mathcal{L}}(F)=0$ and 
\begin{equation*}
\limsup_{x \to \xi} u(x) \leq m \,\, \textup{for all} \,\,\xi \in \d_\infty \Omega \setminus F,
\end{equation*}
then $u \leq m$ in $\Omega$.
\end{lemma}

From this, we get the following lemma.

\begin{lemma}\label{l:sameboundaryval}
Let $\Omega\subset \bR^{d+1}$ be a regular domain and $u,v$ be solutions to $\mathcal{L}$ in $\Omega$. Suppose $\limsup_{X\to x} u(X)\leq \liminf_{X\to x} v(X)$ for all $x\in \d\Omega$ (and, if $d>1$ and $\Omega$ is unbounded, that $\limsup_{X\to \infty} u(X)\leq \liminf_{X\to \infty} v(X)$) , then  $u\leq v$ in $\Omega$. In particular, if $\lim_{X\to x} u(X)= \lim_{X\to x} v(X)$ for all $x\in \d_{\infty}\Omega$, then $u=v$ in $\Omega$. 
\end{lemma}

Indeed, consider $h=u-v$. Then $\limsup_{X\to x}h(X)\leq 0$, and so the previous lemma implies $h\leq 0$ in $\Omega$, and hence $u\leq v$ in $\Omega$. 

%
%


\begin{lemma}[{\cite[Theorem 11.9]{HKM}}]\label{l:Mm}
Let $F$ be a closed subset of $\d_{\infty}\Omega$ where $\Omega\subset \bR^{d+1}$. Let $M\geq m$ and $v(x)$ be a $\mathcal{L}$-subharmonic function such that 
\[
\limsup_{X\rightarrow x\in \d_{\infty}\Omega} v(X) \leq M\one_{E}+m\one_{\d_{\infty}\Omega\backslash E}.\]
Then 
\[
v(X)\leq (M-m)\omega_{\Omega}^{\mathcal{L}}(E)+m \;\; \mbox{ for all }X\in \Omega.\]
\end{lemma}

\begin{lemma}[{\cite[Lemma 11.16]{HKM}}] \label{l:<1}
Suppose that $\d_\infty\Omega$ is regular and that $E$ is a closed subset of $\d_\infty\Omega$. Then $\omega_{\Omega}^{\mathcal{L}}(E)=0$ if and only if $$\sup_{X\in \Omega}\omega_{\Omega}^{\mathcal{L},X}(E)<1.$$
\end{lemma}

\begin{lemma}\label{l:<12}
Suppose that $\d\Omega$ is regular (though perhaps not the point at infinity), $\omega_{\Omega}^{\mathcal{L}}(\infty)=0$,  and $E\subset \d\Omega$ is compact. Then $\omega_{\Omega}^{\mathcal{L}}(E)=0$ if and only if 
$$\sup_{X\in \Omega}\omega_{\Omega}^{\mathcal{L},X}(E)<1.$$
\end{lemma}

\begin{proof}
Note that for any $x\in \d\Omega\backslash E$, 
\[
\lim_{X\rightarrow x}\omega_{\Omega}^{\mathcal{L},X}(E)=0\]
since $x$ is a regular point. For $x\in E$,
\[
\limsup_{X\rightarrow x}\omega_{\Omega}^{\mathcal{L},X}(E)\leq t:=\sup_{X\in \Omega}\omega_{\Omega}^{\mathcal{L},X}(E)<1. \]
Thus, \Lemma{Mm} with $v=\omega_{\Omega}^{\mathcal{L}}$ implies 
\[
\omega_{\Omega}^{\mathcal{L},X}(E)
\leq (t-0)\omega_{\Omega}^{\mathcal{L},X}(E\cup \{\infty\})+0
=t\omega_{\Omega}^{\mathcal{L},X}(E) < \omega_{\Omega}^{\mathcal{L},X}(E)\]
which is a contradiction.
\end{proof}

What will be particularly useful for us about Ahlfors regular NTA domains (aside from being able to construct more NTA regions within) is the following result.

\begin{theorem}[\cite{DJ90,KP01}]
For all $A,C>1$, integers $d\geq 1$, and $\ve>0$, there are constants $C_{DJ}=C_{DJ}(A,C,d)>0$ and $\delta=\delta(\ve,A,C,d)>0$ such that the following holds. Let $\Omega\subset \bR^{d+1}$ be a $C$-NTA domain with an $A$-Ahlfors $d$-regular boundary. Let $B_{0}$ be a all centered on $\d\Omega$, $Z_{0}\in \Omega\backslash C_{DJ}B$,  $\omega=\omega_{\Omega}^{\mathcal{L},Z_{0}}$ where $\mathcal{L}$ satisfies the \hyperref[d:KP]{KP-condition}. Then $\omega$ is $A_{\infty}$-equivalent to $\cH^{d}$ on $B_{0}\cap \d\Omega$, meaning whenever $F\subset B\cap \d\Omega$ with $B\subset B_{0}$ centered on $\d\Omega$, we have
\[
\frac{\omega(F)}{\omega(B)}<\delta \;\; \mbox{ implies } \frac{\cH^{d}|_{\d\Omega}(F)}{\cH^{d}|_{\d\Omega}(B)}<\ve\]
and
\[
 \frac{\cH^{d}|_{\d\Omega}(F)}{\cH^{d}|_{\d\Omega}(B)}<\delta \;\; \mbox{ implies }\frac{\omega(F)}{\omega(B)} <\ve.\]
In particular, $\omega\ll \cH^{d}\ll \omega$ on $\d\Omega_{0}$. 
\label{t:DJ}
\end{theorem}

For the case of harmonic measure, the $d=1$ case is due to Lavrentiev \cite{Lav36}, and to David and Jerison for the case of $d>1$ \cite{DJ90}. In \cite{ABHM15}, it was noted that this more general version holds by a modification using a theorem of Kenig and Pipher. We fill in these details in the appendix.


%

%
%

\subsection{Localization of elliptic measure estimates}

In this section we prove a lemma that will allow us to localize our proofs.
 \def\tom{\widetilde{\Omega}}
\begin{lemma}\label{l:reduction}
Let $\Omega\subset \bR^{d+1}$ be a regular domain, either bounded or such that $\infty$ is Wiener regular. Let $B$ be any ball centered on $\d\Omega$ so that $\Omega$ has the CDC in $2B$. Then there is a bounded open set $\tom\subset \Omega$ such that 
\begin{enumerate}
\item $\tom \supseteq B\cap \Omega$
\item $\d\tom \cap \d\Omega= \cnj{B}\cap \d\Omega$
\item If $\Omega$ has the CDC in $2B$, then $\widetilde \Omega$ and any of its connected components has the CDC.
\item If $\Omega$ has big boundary in $2B$, then $\widetilde \Omega$ and any of its connected components has big boundary.
\end{enumerate}
\end{lemma}

%

%

 \begin{proof}
If $\Omega\subset 2B$, then we just set $\tom=\Omega$ and we are done, so assume $\Omega\backslash 2B\neq\emptyset$.  

Let $C_{1}>1$ be large and $\W'=\W'(\Omega)$ be the set of maximal dyadic cubes $I$ in $\W(\Omega)$ for which $C_{1}I\subset \Omega$. For $\lambda\in (0,1/2)$ small, let
\[
\widetilde{\W}= \{I\in \W':\;\; I\cap B\neq\emptyset\}\]
 and
\[\tom=\bigcup\{ \mbox{int } (1+\lambda)I: I\in \widetilde{\W}\}.\]
Note that for $I\in \widetilde{\W}$, $\dist(I,\partial\Omega)\geq \frac{(C_{1}-1)}{2}\ell(I)$, and since $B$ is centered on $\d\Omega$ and $I$ is maximal, 
\[
\ell(I)
\leq
\frac{2}{C_{1}-1} \dist(I,\d\Omega) 
\leq \frac{2r_{B}}{C_{1}-1}\]
and so for $C_{1}$ large enough,
\begin{equation}
\label{e:tildeom<3/2}
\widetilde{\Omega}\subset \frac{3}{2}B.
\end{equation}

It is also clear that $\tom \supseteq B\cap \Omega$ and $\d\Omega\cap \tom=\cnj{B}\cap \d\Omega$. 

Let $\hat{\Omega}$ be any connected component of $\tom$ or $\tom$ itself. We will show that $\Omega$ having big boundary or the CDC in $2B$ implies $\hat{\Omega}$ has big boundary or the CDC. Let $c=c_{2B\cap \d\Omega}>0$ be as in \Definition{LBB}. Let $x\in \d\hat{\Omega}$ and 
\begin{equation}\label{e:r<3rb}
0<r<\diam \d \hat{\Omega}\stackrel{\eqn{tildeom<3/2}}{\leq} \diam \tom \leq 3r_{B}.
\end{equation}
We consider the  following two cases $r>10\dist(x,\d\Omega)$ and $r<10\dist(x,\d\Omega)$.  

 \begin{enumerate}
 \item[(1)] Let $r>10\dist(x,\d\Omega)$.
\end{enumerate} 
In this case there is $y\in \d\Omega\cap B(x,r/10)$ and since $x\in \d \widetilde{\Omega}$ and $B\cap \Omega\subseteq \widetilde{\Omega}\subseteq  \frac{3}{2}\cnj{B}$ it follows that
\begin{align*}
\dist(y,B) &   \leq |y-x|+\frac{r_{B}}{2}  
<\frac{r}{10}+\frac{r_{B}}{2} \stackrel{\eqn{r<3rb}}{<} \frac{3r_{B}}{10}+\frac{r_{B}}{2} =\frac{4}{5}r_{B}.
\end{align*}
Hence we know $y\in B(y,r_{B}/5)\subset \d\Omega\cap 2B$ thereupon we conclude that 
 \[
 \cH^{d}_{\infty} (B(x,r)\backslash \hat{\Omega})
 \geq \cH^{d}_{\infty}(B(x,r)\backslash \Omega)
 \geq \cH^{d}_{\infty}(B(y,r/15)\backslash \Omega)
 \geq c(r/2)^{d},\]
 where we used that $B(y,r/15) \subset B(y, r_B/5) \subset 2B$. 
 
 Similarly, if the CDC holds in $2B$, then since $B(y,r/5)\subset B(x,2r)$
  \begin{align*}
  \textup{cap} (B(x,r)\backslash \hat{\Omega},B(x,2r))
&  \geq   \textup{cap}(B(x,r)\backslash \Omega,B(x,2r))\\
&  \geq   \textup{cap}(B(y,r/15)\backslash \Omega,B(x,2r)))\\
&  \gec \textup{cap}(B(y,r/15)\backslash \Omega,B(y,2r/15)))\\
&  \gec (r/15)^{d-1}.\end{align*}
 
\begin{enumerate}
\item[(2)] Let $r<10\dist(x,\d\Omega)$.
\end{enumerate}
 In this case $x\in \d\hat{\Omega}\backslash \d\Omega$ and there is $I\in \widetilde{\W}$ so that $x\in \d (1+\lambda) I$. Then for $\lambda>0$ small enough, $x$ is contained in a $d$-dimensional rectangle $R$ in $\d(1+\lambda )I\cap \d\hat{\Omega}$ with sidelengths comparable to $\ell(I)^{d}$. Thus,
\[
\cH_{\infty}^{d}(B(x,r)\cap \d\hat{\Omega})
\geq \cH^{d}_{\infty}(B(x,\dist(x,\d\Omega)/2)\cap R)
\gec \dist(x,\d\Omega)^{d}
\gec r^{d}.\]
Moreover, $\textup{cap} (B(x,r)\backslash \hat{\Omega},B(x,2r))\gec r^{d-1}$ by Theorem \ref{thm:Martio}.

 This proves that each component has big boundary if $\Omega$ has big boundary in $2B$, and the CDC if $\Omega$ has the CDC in $2B$. The proof that $\Omega$ having the CDC in $2B$ implies $\widetilde \Omega$ has the CDC is similar and leave the details to the reader (since all we have used about $\cH^{d}_{\infty}$ is that it is monotone, and the same goes for capacity).\\

\end{proof}

Here we develop a local version of \Lemma{bourgain}.

\begin{lemma}\label{l:localbourgain}
Let $\Omega\subset \bR^{d+1}$ be a regular domain, either bounded or such that $\infty$ is Wiener regular. Let $B$ be any ball centered on $\d\Omega$ so that $\Omega$ has the CDC in $2B$ and $\d\Omega\backslash 2B\neq\emptyset$. Then
\begin{equation}\label{e:localbourgain}
\omega_{\Omega}^{X}(2B)\gec 1 \;\; \mbox{ for all }X\in \Omega\cap B.\end{equation}
\end{lemma}

\begin{proof}
Let $\tom\subset \Omega$ be as in \Lemma{reduction} for the ball $2B$, so $\tom$ has the CDC. Then for $X\in B$, by \hyperref[l:CP]{Carleman's Principle} and \Lemma{bourgain}
\[
\omega_{\Omega}^{X}(2B)
\geq \omega_{\tom}^{X}(2B)
\stackrel{\eqn{bourgain}}{\gec}
1.\]
\end{proof}

A consequence of this is the following lemma, which says that if a point in $\Omega$ is close to a point in the interior of a set $F\subset \Omega$, then $\omega_{\Omega}^{X}(F)$ is large.

\begin{lemma}\label{l:chain}
Let $\Omega\subset \bR^{d+1}$ be an open set with the CDC in a ball $2B_{0}$  centered on $\d\Omega$. Let $F\subset \d\Omega$, and $\{B_{j}\}_{j=1}^{N+1}$ are a sequence of balls such that, for some $c>1$, 
\begin{enumerate}
\item $cB_{j}\subset B_{0}$ for all $j=1,...,N+1$,
\item $B_{j}\cap B_{j+1}\neq\emptyset$ for all $j=1,...,N$,
\item $cB_{j}\subset \Omega$ for $j=1,...,N$, 
\item $B_{N+1}\cap \d\Omega\neq\emptyset$,
\item $cB_{N+1}\cap \d\Omega\subset F$.
\end{enumerate}
Let $Y_{j}$ be the centers of the $B_{j}$. Then 
\[
\omega_{\Omega}^{\mathcal{L},Y_{j}}(F)\gec_{c,N}1 \;\; \mbox{ for all } 1\leq j\leq N.
\]
\end{lemma}

\begin{proof}
Let $\omega^{X}= \omega_{\Omega}^{\mathcal{L},X}$. By adjusting or replacing our balls if necessary, we can assume without loss of generality that our balls also satisfy
\begin{equation*}
\diam (\d B_{N}\cap \d B_{N+1})\sim r_{B_{N}}.
\end{equation*}

There is a ball $B'\subset B_{N}\cup B_{N+1}\cap \Omega$ so that 
\[\d B_{N}\cap \d B_{N+1} \subset \d B'\] 
and $\d B'\cap \d\Omega\neq\varnothing$, see Figure \ref{f:bj}. 
\begin{figure}[!ht]
\includegraphics[width=250pt]{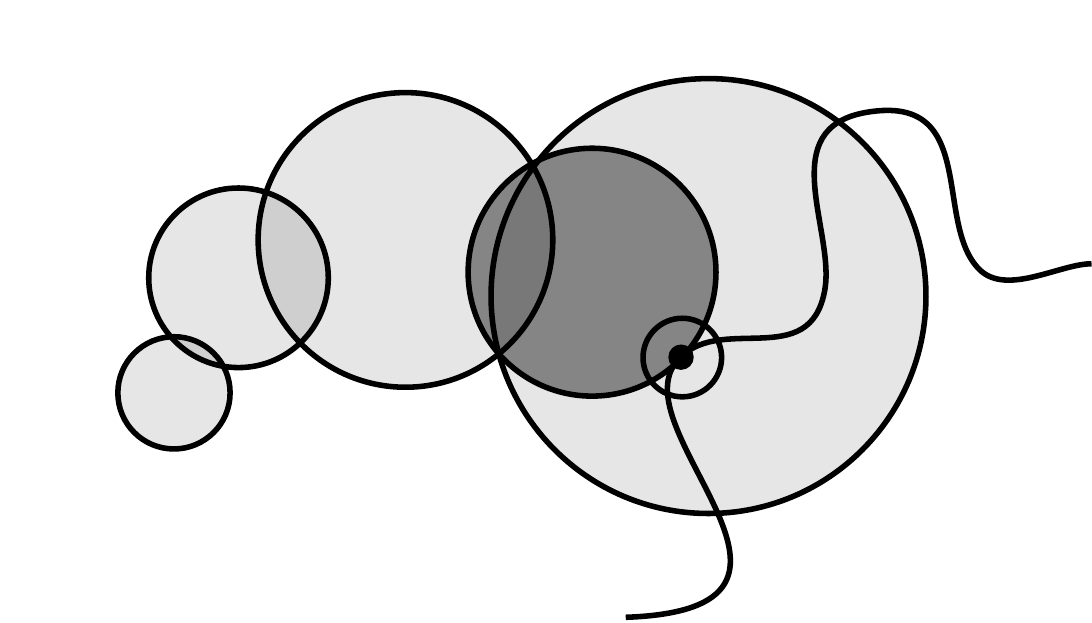}
\begin{picture}(0,0)(250,0)
\put(30,50){$B_{1}$}
\put(75,100){$B_{N}$}
\put(150,110){$B_{N+1}$}
\put(135,75){$B'$}
\put(160,50){$B''$}
\put(120,0){$\d\Omega$}
\end{picture}
\caption{}
\label{f:bj}
\end{figure}
To find this, note that the centers of all balls in $\bR^{d+1}$ that contain $\d B_{N}\cap \d B_{N+1}$ in their boundaries are exactly the infinite line passing through the centers of $B_{N}$ and $B_{N+1}$. Moreover, the ones with centers on the segment between the centers of $B_{N}$ and $B_{N+1}$ are contained in $B_{N}\cup B_{N+1}$. As $B_{N}\subset \Omega$ and $B_{N+1}\not\subset \Omega$, there is a ball $B'$ whose center is on this segment such that $\d B'\cap \d\Omega\neq\varnothing$.

Let 
\[\zeta\in \d B'\cap \d\Omega\cap B_{N+1},\]
and
\[B''=B(\zeta,(c-1)r_{B_{N+1}}) \subset cB_{N+1}.\]
In particular, 
\[B''\cap \d\Omega\subset c B_{N+1} \cap \d\Omega \subset F.\]

Now let $B'''\subset \frac{1}{2} B''\cap B'$ be a ball of radius $ r_{B''}/4$ and let ${\mathcal{Z}} \in \Omega$ denote its center. Since $B_{j}$ is a Harnack chain, $r_{B_{N}}\sim r_{B_{N+1}}$, and by assumption, 
\begin{align*}
r_{B_{N}} &\sim \diam (\d B_{N}\cap \d B_{N+1})\leq \diam B' \\
&\leq \diam B_{N}+\diam B_{N+1}\lec \diam B_{N}.
\end{align*}
Thus, $B_{N}\cup B'$ is itself an NTA domain. Note that by assumptions (2) and (3) and Harnack's inequality, $\omega^{Y_{j}}(F)\sim \omega^{Y_{j+1}}(F)$ for all $j=1,...,N$. Hence, by the Harnack chain condition inside $B_{N}\cup B'$, repeated use of Harnack's inequality on the sets $B_{j}$ for $j=1,...,N$, (1), \Lemma{localbourgain}, and the facts that $B'' \subset \Omega_i$ and ${\mathcal{Z}}\in \frac{1}{2} B''$, we have for $1\leq j\leq N$ that
\begin{align*}
\omega^{Y_{j}}(F)
& \gec_{C,N} \omega^{Y_{N}}(F)
\gec_{\delta} \omega^{\mathcal{Z}}(F)
\geq \omega^{\mathcal{Z}}(B''\cap \d\Omega)
\stackrel{(1) \atop \eqn{localbourgain}}{\gec} 1.
\end{align*}
\end{proof}

\begin{lemma}\label{l:reduction2}
Let $\Omega\subset \bR^{d+1}$ be a regular domain. If $d=1$, assume $\omega_{\Omega}^{\mathcal{L}}(\infty)=0$ or $\infty$ is regular. Let $B$ be any ball centered on $\d\Omega$ so that $\Omega$ has the CDC (or big boundary) in $2B$. If $E\subset B\cap \d\Omega$ is a Borel set with $\omega_{\Omega}^{\mathcal{L},X}(E)>0$ for some $X\in \Omega$, then there is a connected open set $\hat{\Omega}\subset \Omega$ with the CDC (or big boundary) and contained in the component of $\Omega$ containing $X$. There is also $\hat{X}\in \hat{\Omega}$ for which $E\subset \d\Omega\cap \d\hat{\Omega}$ and $\omega_{\hat{\Omega}}^{\mathcal{L},\hat{X}}(E)>0$. 
\end{lemma}

\begin{proof}
To simplify notation, we write $\omega_{\Omega}^{\mathcal{L}}$ as $\omega_{\Omega}$. By inner regularity of harmonic measure, we may assume that $E$ is compact and $\omega_{\Omega}^{X}(E)\in (0,1)$ for some $X\in \Omega$. Without loss of generality, we may assume $\Omega$ is connected, as it is easy to check that this component will also have the CDC or big boundary in $B$ if $\Omega$ does.

%
Let
\[
D_{1}=\{Z\in \d\tom: \dist(Z,\d\Omega)\geq t/2\},\;\; D_{2}=\d\tom\cap \Omega\backslash D_1.\]

By the compactness of $D_{1}$,  
\begin{equation}\label{e:infzind1}
\inf_{Z\in D_{1}}\omega_{\Omega}(E^{c})>0.
\end{equation}
Let $Z\in D_{2}$ and let $Z'\in \d\Omega$ be the closest point in $\d\Omega$ to $Z$, so that $|Z-Z'|<t/2$. Then $B(Z',t)\subseteq E^{c}$, and so
\[
\omega_{\Omega}^{Z}(E^{c})
\geq \omega_{\Omega}^{Z}(B(Z',t))\stackrel{\eqn{localbourgain}}{\gec} 1 \;\; \mbox{ for all }Z\in D_{2}.\]
This and \eqn{infzind1} imply

\begin{equation}\label{e:omegat1}
s:=\sup_{Z\in \d \tom \cap \Omega} \omega_{\Omega}^{Z}(E)<1.
\end{equation}

If $\omega_{\hat{\Omega}}^{X_{0}}(E)=0$ for all $X_{0}\in \hat{\Omega}$, then by Lemma \hyperref[l:identity-hm]{B.1} in the appendix, for any $X_{0}\in \tom\cap \Omega$,

\begin{equation}\label{e:omegat2}
\omega_{\Omega}^{X_{0}}(E)
=\omega_{\tom}^{X_{0}}(E)+ \int_{\d\tom\cap \Omega} \omega_{\Omega}^{X}(E) d\omega_{\tom}^{X_{0}}(X) 
<0+s=s.
\end{equation}

Because each $z\in \d\Omega\backslash \d{\tom}$ is a regular point and $E\subset \d\tom \cap \d\Omega$, we know 
\begin{equation}\label{e:omegat3}
\lim_{\Omega\ni Z \rightarrow z}\omega_{\Omega}^{Z}(E)=0 \;\; \mbox{ for all $z\in \d\Omega\backslash \d{\tom}$}.\end{equation}
If $\omega_{\Omega}^{\mathcal{L}}(\infty)=0$, we can use \eqn{omegat1}, \eqn{omegat3},  and \Lemma{MP} (using the fact that $\omega_{\Omega}(\infty)=0$) to get $\omega_{\Omega}^{X_{0}}(E)\leq s$ for all $X_{0}\in \Omega\backslash \tom$. If $\infty$ is regular, then $ \omega_{\Omega}^{Z}(E)\rightarrow 0$ as $Z\rightarrow \infty$, and so we can  use \Lemma{MP} again to conclude still that $\omega_{\Omega}^{X_{0}}(E)\leq s$ for all $X_{0}\in \Omega\backslash \tom$.

Combining this with \eqn{omegat2}, we know that  $\omega_{\Omega}^{X_{0}}(E)\leq s$ for all $X_{0}\in \Omega$, which by \Lemma{<1} implies $\omega_{\Omega}^{X_{0}}(E)=0$ for all $X_{0}\in \Omega$, which is a contradiction. Thus, there is $X\in \tom$ such that $\omega^{X}_{\tom}(E)>0$. If we set $\hat{\Omega}$ to be the component of $\tom$ containing this $X$, then 
\[\omega_{\hat{\Omega}}^{X}(E)=\omega_{\tom}^{X}(E)>0.\]

\end{proof}

 \section{The Main Lemma and the Proof of \TheoremI}
\label{sec:mainthm}
\label{s:I}

In this section, we will drop the dependence on $\mathcal{L}$ and let $\omega_{\Omega}^{X}$ denote any elliptic measure satisfying the \hyperref[d:KP]{KP-condition}.\\ 

The objective of this section is to prove the following lemma.

\begin{lemi}\label{l:lemI}
\assumption  If $d=1$, assume $\omega_{\Omega}(\infty)=0$ or $\infty$ is regular. Suppose $\Gamma\subset \bR^{d+1}$ is $A$-Ahlfors $d-$regular and splits $\bR^{d+1}$ into two NTA domains $\Omega_{1}$ and $\Omega_{2}$. If $E\subset \d\Omega\cap \Gamma\cap B_{0}$ is a Borel set with $\omega_\Omega^{X_{0}}(E)>0$, then there are $X_{i}\in \Omega_{i}\cap \Omega$ for $i=1,2$ in the same component of $\Omega$ as $X_{0}$ so that $\omega_{\Omega\cap \Omega_{1}}^{X_{1}}(E)+\omega_{\Omega\cap \Omega_{2}}^{X_{2}}(E)>0$.
\end{lemi}

\TheoremI follows immediately since \Theorem{DJ} implies that if $\cH^{d}(E)=0$ for some $E\subset \Gamma\cap \d\Omega$, then $\omega_{\Omega\cap \Omega_{i}}^{X_{i}}(E)\leq \omega_{ \Omega_{i}}^{X_{i}}(E) =0$ for $i\in \{1,2\}$, and then Lemma \hyperref[l:lemI]{I } implies $\omega_{\Omega}^{X_{0}}(E)=0$. \\

For the remainder of this section, we focus on proving Lemma \hyperref[l:lemI]{I}. The beginning of the proof follows that in \cite{Wu86}, which in turn has its roots in \cite{McM69}, but then we take a large departure at around the time Wu uses the exterior corkscrew condition, which we are not assuming to hold for $\Omega$.\\

We claim that it suffices to prove the lemma for the case that $\Omega$ is connected and bounded with big boundary. Indeed, if we prove this case, then for the general case, we just need to pick a ball $B_{0}'$ with $2B_{0}'\subseteq B_{0}$ and $\omega_{\Omega}^{X}(E\cap B_{0}')>0$, then \Lemma{reduction2} implies we may find $\hat{\Omega}\subseteq \Omega$ connected, bounded, in the same component of $\Omega$ containing $X_{0}$, and with big boundary so that $E\cap B_{0}'\subseteq \d\hat{\Omega}$ and $\omega_{\hat{\Omega}}^{X}(E)>0$ for some $x\in \hat{\Omega}$. Then, assuming we can prove the lemma for this case, there is $i\in \{1,2\}$ and $X_{i}\in \Omega_{i}\cap \tilde{\Omega}$ so that $0<\omega_{\Omega_{i}\cap \hat{\Omega}}^{X_{i}}(E)$, and then $0<\omega_{\Omega_{i}\cap {\Omega}}^{X_{i}}(E)$ by  \hyperref[l:CP]{Carleman's Principle}. 

Thus, without loss of generality we may assume that $\Omega$ is bounded and has big boundary.\\
 
Let $\Gamma$ and $\Omega$ be as in Lemma \hyperref[l:lemI]{I}. Let $E\subset\partial\Omega\cap\Gamma$ be a Borel set with $\omega^{X_{0}}_{\Omega}(E)>0$ but
\begin{equation}\label{e:ww=0}
\omega_{\Omega\cap \Omega_{1}}^{X_{1}}(E)+\omega_{\Omega\cap \Omega_{2}}^{X_{2}}(E)=0 \;\; \mbox{ for all $X_{i}\in \Omega_{i}\cap \Omega$, $i=1,2$.}
\end{equation}
 Our goal now is to show that there is $\gamma\in (0,1)$ so that
\begin{equation}\label{e:<gamma}
\omega_{\Omega}^{X}(E)<\gamma \;\; \mbox{ for }\;\; X\in \Gamma\cap \Omega.
\end{equation}
If this is the case, then by Lemma \hyperref[l:identity-hm]{B.1} in the appendix, if $X_{0}\in \Omega\cap \Omega_{1}$, 

\begin{align*}
\omega^{X_{0}}_{\Omega}(E)
& =\omega_{\Omega\cap \Omega_{1}}^{X_{0}}(E)+\int_{\Gamma\cap \Omega} \omega^{X}_{\Omega}(E)d\omega^{X_{0}}_{\Omega\cap \Omega_{1}}(X)
 \stackrel{\eqn{ww=0}\atop \eqn{<gamma}}{<} 0+\gamma=\gamma<1
\end{align*} 
Similarly, we have that $\omega^{X_{0}}_{\Omega}(E)<\gamma<1$ for all $X_{0}\in \Omega\cap \Omega_{2}$, which along with \eqref{e:<gamma} implies $\omega^{X_{0}}_{\Omega}(E)<\gamma<1$ for all $X_{0}\in \Omega$. 
Since $\Omega$ is regular, by \Lemma{<1}, $\omega^{X_{0}}_{\Omega}(E)=0$ for all $X_{0}\in \Omega$ { (for this we have to assume $E$ is closed, but $\omega^{X_{0}}_{\Omega}(E')<\gamma$ for any closed subset $E'\subset E$, and so we still get $\omega^{X_{0}}_{\Omega}(E)=0$ by inner regularity of harmonic measure)}, and so we get a contradiction, proving the theorem.\\

Now we focus on proving \eqn{<gamma}. Let $X\in \Gamma\cap \Omega$ and $r=\dist(X,\d\Omega).$ Since $\Omega_{i}$ are $C$-NTA domains, if $c=C^{-1}$, there are balls 
\[
B(Y^{i},cr)\subset \Omega_{i}\cap B(X,r)\, \, \mbox{for}\, \, i=1,2.\] 
Let 
\begin{align}
\label{e:balli}
B^{i}= B\ps{Y^{i},\frac{cr}{2}}.
\end{align}
We claim it is enough to show that there is $\eta\in (0,1)$ so that
\begin{equation}\label{e:<eta}
\min_{i=1,2}\sup_{Y\in B^{i}}\omega_{\Omega\cap \Omega_{i}}^{Y}(\Gamma\cap \Omega)<\eta.
\end{equation}
Indeed, note that by the Harnack chain condition, there is $t\in (0,1)$ depending only on $C$ so that
\[
\omega_{\Omega}^{X}(E^{c})>t\omega_{\Omega}^{Y}(E^{c}) \mbox{ for all }Y\in B^{i}.
\]
Hence, it follows that for all $Y\in B_{i}$, $i=1,2,$ we have
\begin{align}\label{e:<1-t}
\omega_{\Omega}^{X}(E) &=1-\omega_{\Omega}^{X}(E^{c}) <1-t\omega_{\Omega}^{Y}(E^{c})\\ \notag
&=(1-t)+t\omega_{\Omega}^{Y}(E).
\end{align}
If the minimum in \eqref{e:<eta} is attained for $i=1$, then if $Y \in B_1$, by \eqref{e:<1-t} and \eqref{e:identity-hm}, we have that
\begin{align*}
\omega^{X}_{\Omega}(E)
& 
<(1-t)+t\omega_{\Omega}^{Y}(E)\\
& =(1-t)+t\ps{\omega_{\Omega\cap \Omega_{1}}^{Y}(E)+\int_{\Gamma\cap \Omega} \omega^{Z}_{\Omega}(E)d\omega^{Y}_{\Omega\cap \Omega_{1}}(Z)}\\
& \stackrel{\eqn{ww=0}\atop \eqn{<eta}}{<}
(1-t)+t\ps{0+ \eta}=(1-t)+t\eta<1.
\end{align*} 
The same holds if the minimum in \eqref{e:<eta} is attained for $i=2$. Thus, this finishes the proof of \eqn{<gamma} and the claim. We now focus on showing \eqn{<eta}. \\

For $i=1,2,$ we denote
\[
\Omega^{i}=\Omega_{i}\cap \Omega\]
and
\[\omega_{i}=\omega_{\Omega^{i}}.\]
Note that
\begin{align}
\d\Omega^{i}=\d(\Omega_{i}\cap \Omega) 
& =(\d\Omega_{i} \cap \Omega)\cup (\d\Omega\cap \Omega_{i})\cup (\d\Omega_{i}\cap \d\Omega) \notag \\
& =(\Gamma\cap \Omega)\cup (\d\Omega\cap \Omega_{i}) \cup  (\Gamma\cap \d\Omega)
\label{e:omegaidecomp}
\end{align}
Indeed, let $x\in \bR^{d+1}$. We split into three cases.
\begin{enumerate}
\item Suppose first that $B(x,r)\cap \Omega^{c}=\emptyset$ for some $r>0$, then $B(x,r)\subseteq \Omega$ and so for all $r'<r$, $B(x,r')\cap \Omega^{i}=B(x,r')\cap \Omega_{i}$ and $B(x,r')\backslash \Omega^{i}=B(x,r')\backslash \Omega_{i}$, hence $x\in \d\Omega^{i}$ if and only if $x\in \d\Omega_{i}$, and since $x\in \Omega$ in this case, this is true if and only if $x\in \d\Omega_{i}\cap \Omega$. \item Similarly, if there is $r>0$ so that $B(x,r)\cap \Omega_{i}^{c}=\emptyset$, then $x\in \d\Omega\cap \Omega_{i}$ if and only if $x\in \d\Omega^{i}$. 
\item If $B(x,r)\cap \Omega_{i}^{c}\neq\emptyset$ and $B(x,r)\cap \Omega^{c}\neq\emptyset$ for all $r>0$, then 
\begin{equation}
\label{e:xincomcomi}
x\in \cnj{\Omega_{i}^{c}}\cap \cnj{\Omega^{c}}.
\end{equation}
 If $x\in \d\Omega^{i}$, then $x\in \cnj{\Omega}\cap \cnj{\Omega_{i}}$, and \eqn{xincomcomi} implies $x\in \d\Omega\cap \d\Omega_{i}$. Conversely, if $x\in \d\Omega\cap \d\Omega_{i}$, then $B(x,r)\cap \Omega\neq\emptyset$ and $B(x,r)\cap \Omega_{i}\neq\emptyset$ for all $r>0$ and so $x\in \d\Omega^{i}$. 
 \end{enumerate}
 This completes the proof of \eqn{omegaidecomp}.\\

A consequence of \eqn{omegaidecomp} is that
\begin{equation}\label{e:boundary}
\d\Omega^{i}\backslash (\Gamma\cap \Omega) = (\d\Omega\cap \Omega_{i})  \cup (\Gamma\cap \d\Omega).
\end{equation}

Let $Y^{i}$ be the center of $B^{i}$ where $B^{i}$ is as in \eqn{balli}. To show \eqn{<eta}, by Harnack chains, it suffices to show that
\begin{equation}\label{e:eta2}
\min_{i=1,2} \omega_{i}^{Y^{i}}(\d\Omega^{i}\backslash (\Gamma\cap \Omega))\gec 1.
\end{equation}

Let $M_{0}>2$ to be decided later and recall that $X \in \Gamma \cap \Omega$, $r=\dist(X, \partial \Omega)$, and $Y^{1}$ is the center of $B^{1}$ defined as in \eqn{balli}. Suppose that 
\begin{equation}\label{e:notclose}
\mbox{ there is }Z\in \d\Omega\cap B(X,M_{0}r)\cap \Omega_{1}\mbox{ so that }\dist(Z,\Gamma)\geq \ve r,
\end{equation}
see Figure \ref{f:case1}.
\begin{figure}[!ht]
\includegraphics[width=300pt]{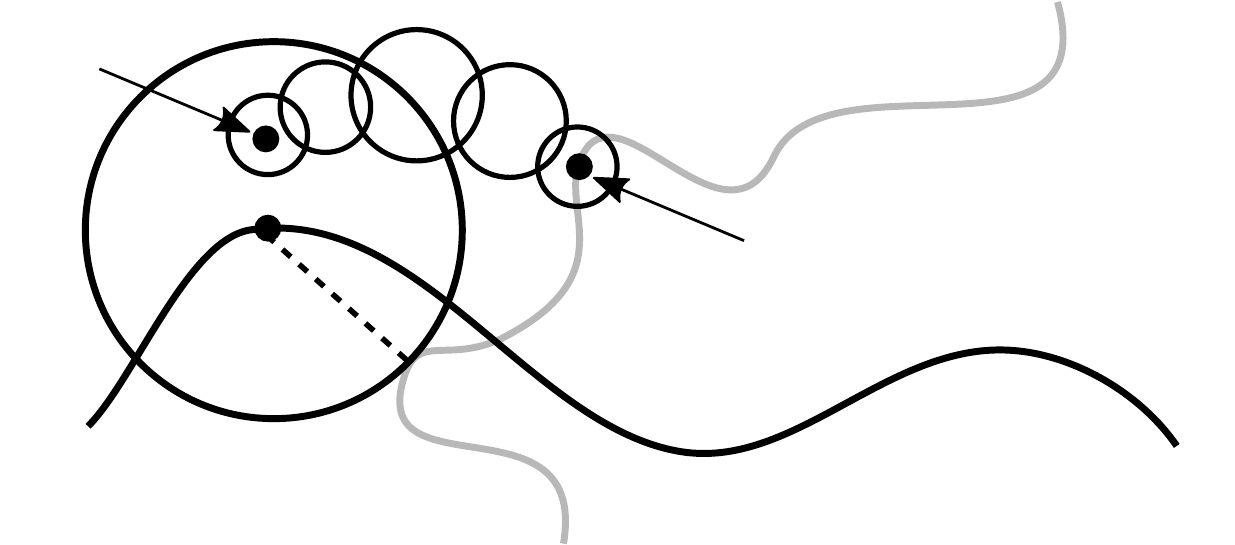}
\begin{picture}(0,0)(300,0)
\put(225,50){$\Gamma$}
\put(225,120){$\d\Omega$}
\put(75,45){$r$}
\put(50,60){$X$}
\put(5,110){$Y^{1}$}
\put(180,70){$Z$}
\put(100,127){$B_{j}$}
\end{picture}
\caption{}
\label{f:case1}
\end{figure}

Let now $\{B_{i}\}_{i=1}^{N}$ be a Harnack chain in $\Omega_{1}$ from the center $Y^{1}$ of $B^{1}$ to $Z$, so $N\lec_{M_{0},C,\ve} 1$. Let $j$ be the smallest integer for which $B_{j}\cap \d\Omega\neq\emptyset$. By \Lemma{chain}, we have 

\[
\omega^{Y^{1}}_{1}(\d\Omega^{1}\backslash (\Gamma\cap \Omega))
\gec_{C,\ve,M_{0}} 1.\]
Thus, \eqn{eta2} holds in this case (i.e. when \eqn{notclose} holds) and we can repeat the same for when \eqn{notclose} holds with $\Omega_{2}$ in place of $\Omega_{1}$ to get the same result. %
Hence, from now on, we will assume instead of \eqn{notclose} that
\begin{equation}\label{e:close}
\dist(Z,\Gamma)< \ve r \mbox{ for all }Z\in \d\Omega \cap B(X,M_{0}r) \cap \Omega_{1}.
\end{equation} 
$x_{0}\in \d\Omega\cap \d B(X,r)$. Let $\bD=\bD(\Gamma)$ be dyadic subdivision of $\Gamma$.  We can arrange our cubes so that there is $Q_{0}\in \bD$ with the property that (using \eqn{close})
\[
B(X,2r)\cap \Gamma\subset \Delta(x_{Q_{0}},Cr_{Q_{0}}), \;\; r\sim r_{Q_{0}},\;\;  \mbox{ and }x_{Q_{0}}\in B(x_{0},\ve r) \]
where $\Delta(x_{Q_{0}},Cr_{Q_{0}})$ is a surface ball.  By \eqn{dcontent}, and for $\ve>0$ small enough, we have
\[
\cH^{d}_{\infty}( B_{Q_{0}}'\backslash \Omega )\gec r_{B_{Q_{0}}'}^{d} \sim r^{d},\]
where $B_{Q_{0}}'$ is as in \Lemma{HM3}. Since $\bR^{d+1}=\cnj{\Omega}_{1} \cup \cnj{\Omega}_{2}$, we may assume without loss of generality that 
\begin{equation}\label{e:om1content}
\cH^{d}_{\infty}(\Omega^{c} \cap \cnj{\Omega}_{1}\cap B_{Q_{0}}')\gec r^{d}.\end{equation}

\def\T{\mathcal{T}}

We now pick $M_{0}$ large enough (depending only on the NTA constants and $d$) so that $B(X,M_{0}r)\supseteq T_{Q_{0}}$ and $\ve$ small enough so that by \eqn{close}
\begin{equation*}
U_{Q_{0}}\cap \d\Omega=\emptyset
\end{equation*}
(where  $T_{Q_{0}}$ is defined right after \eqn{XQ} and $U_{Q_{0}}$ is as in \eqn{whitney3}). It is not hard to check that under our assumptions it holds that $U_{Q_{0}}\subset \Omega$. Indeed, since $\Omega_{1}$ is NTA, we may find a path in $B(X,M_{0}r)\cap \Omega_{1}$ between $Y^{1}$ and $U_{Q_{0}}$ that is at least $cr>0$ away from $\d\Omega_{1}=\Gamma$ where $c$ depends on the NTA constant, so for $\ve>0$ small enough (depending on $c$), this path will avoid $\d\Omega$. Since $Y^{1}\in \Omega$, this means $U_{Q_{0}}\subseteq \Omega$ as well. 

Let $\F$ be the (disjoint) maximal cubes $Q\subset Q_{0}$ for which 
\[
U_{Q}\subset \Omega\cap \Omega_{1}\]
but there is a child $Q'$ of $Q$ for which 
\[
U_{Q'}\not\subset \Omega\cap \Omega_{1}.\]
Maximality of cubes in $\F$ and $U_{Q'}\subset \Omega_{1}$ imply that
\[
U_{Q'}\cap \d\Omega\neq\emptyset.
\] 
Also let $\T\subset \bD_{\F,  Q_0}$ be the maximal cubes contained in $Q_{0}\backslash \d\Omega$ which do not contain any cubes from $\F$. Note that by definition $\bD_{\F, Q_0}$ is the local discretized sawtooth region relative to $\F$ which is the collection of cubes in $\bD_{Q}$ that are not contained in any $Q\in \F$. Therefore, cubes in $\F \cup\T$ forms a disjoint family. Let $\F'$ be all children of cubes in $\F\cup \T$, so that 
\[
\bD_{\F', { Q_0}}=\{Q:Q\supseteq R\mbox{ for some }R\in \F\cup \T\}.\] 
Now set
\[
\Omega'=\Omega_{\F',Q_{0}}\subset \Omega_{1}\]
where $\Omega_{\F',Q_{0}}$ is the local sawtooth region relative to $\F'$ in $Q_0$ as defined right after \eqn{XQ}.

Note that
\begin{equation}\label{e:UinO}
U_{Q}\subset \Omega \;\; \mbox{ for } \;\; Q\in \bD_{\F', Q_0}
\end{equation}
and thus 
\[
\Omega'\subset \Omega\cap\Omega_{1}.\]


\begin{lemma}
Let $\Omega'$ be as above. Then $\Omega'\neq\emptyset$.
\end{lemma}

\begin{proof}
In case $\F=\emptyset$ then $U_{Q}\cap \d\Omega=\emptyset$ for all $Q\in \bD_{Q_{0}}$, and so
\[
B_{Q_{0}}'\cap \d\Omega\cap \Omega_{1}
\subseteq B_{Q_{0}}\cap \Omega_{1}
\subseteq \Omega_{\F,Q_{0}}\subseteq (\d\Omega)^{c}\]
Hence, $B_{Q_{0}}'\cap \d\Omega\cap \Omega_{1}=\emptyset$, which contradicts \eqn{om1content}. Thus, $\F\neq\emptyset$, which implies that $Q_{0}\not\in \T$ since $Q_{0}$ trivially contains a cube in $\F$. Hence, $Q_{0}\not\in \F'$, and so $\Omega'\neq\emptyset$.
\end{proof}

See Figure \ref{f:tw} for the rest of this proof.  

\begin{figure}[!ht]
\includegraphics[width=300pt]{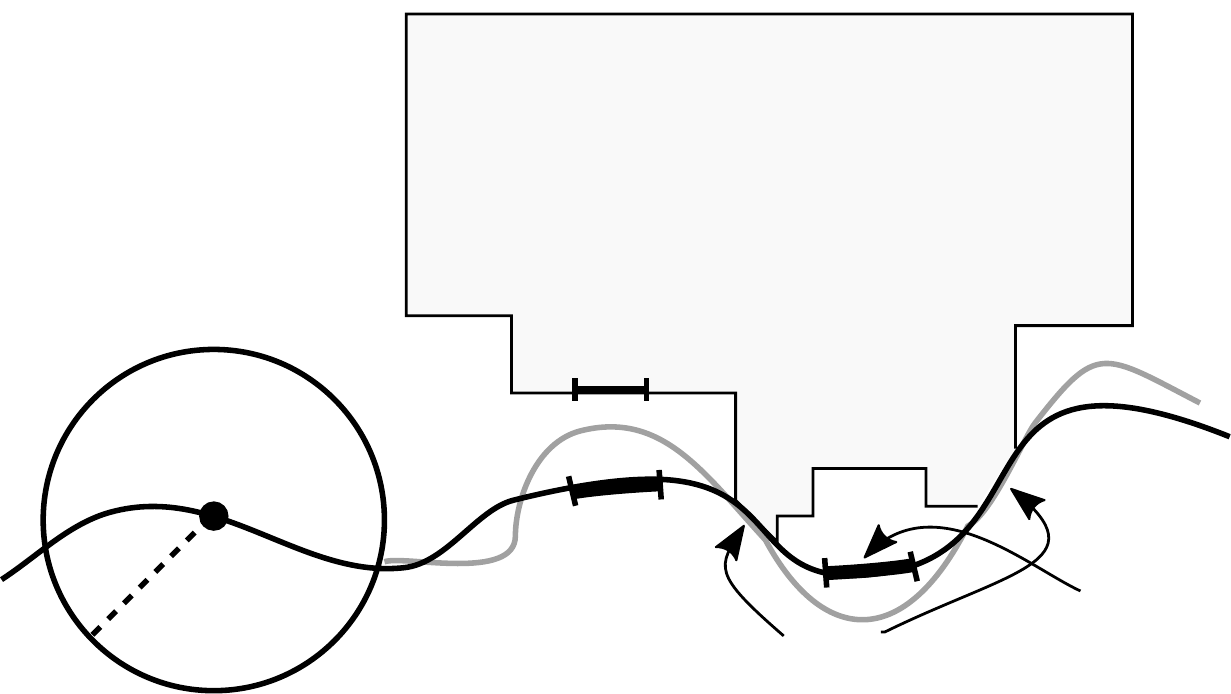}
\begin{picture}(0,0)(300,0)
\put(200,125){$\Omega'$}
\put(40,25){$X$}
\put(130,35){$Q\in \F$}
\put(20,25){$r$}
\put(137,80){$B^{Q}$}
\put(275,50){$\Gamma$}
\put(280,80){$\d\Omega$}
\put(175,0){$\d\Omega'\cap \d\Omega$}
\put(250,15){$R\in \T$}
\end{picture}
\caption{The sawtooth region $\Omega'$ is constructed by adding Whitney regions $U_{Q}$ which do not intersect $\d\Omega$ (corresponding to cubes in $\F$) and which do not get too close to large gaps in $\Gamma\backslash \d\Omega$ (corresponding to cubes in $\T$).}
\label{f:tw}
\end{figure}

For the next lemma, let $\Omega_{Q}^{*}$ be as in remark \ref{r:fatstar} relative $Q$.

\begin{lemma}\label{l:3.2}
For every $Q\in \F$ there is a ball $B^{Q}$ centered on $\d\Omega'$ so that
\begin{equation}
B^{Q}\subset \Omega_{Q}^{*},
\label{e:BinQ}
\end{equation}
\begin{equation}\label{e:rB}
r_{B^{Q}}\sim \ell(Q),
\end{equation}

\begin{equation}
\sum_{Q\in \F} \one_{B^{Q}}\lec \one_{\Omega_{1}},
\label{e:sumB}
\end{equation}

and 
\begin{equation}
\omega_{1}^{Y}(\d\Omega^{1}\backslash (\Gamma\cap \Omega))\gec \one_{B^{Q}\cap \d\Omega'}.
\label{e:>B}
\end{equation}
\end{lemma}
\begin{proof}
  For $Q\in \bD$, let
 $\F_{Q}$ be the collection of children of $Q$ and $\F_{Q}'$ all the grandchildren of $Q$ (so that $\bD_{\F_{Q}',Q}=\{Q\}\cup \F_{Q}$). Set
\[
\Omega_{Q}=\Omega_{\F_{Q}',Q}=  U_{Q} \cup  \bigcup_{R \in \F_{Q}} U_{R}.\]
{  By \Lemma{HM1} and \Lemma{HM2}, $\Omega_{Q}$ is also an NTA domain with Ahlfors regular boundary and $\Omega_{Q}\subset \Omega_{1}$. }

Note that if $Q\in \F$, there is a child $Q'$ of $Q$ so that $U_{Q'}\cap \d\Omega\neq\emptyset$, hence $\Omega_{Q}\cap \d\Omega\neq\emptyset$. Let $\eta>0$ be small enough so that
\begin{equation}\label{e:oqfar}
\dist(\Omega_{Q},\Gamma)>4\eta\ell(Q)  \mbox{ for all }Q\in \bD.\end{equation}

For $Q\in \F$, we pick $B^{Q}$ as follows.\\

{\bf Case 1:} Suppose 

\begin{equation*}
\dist(\d\Omega'\cap \Omega_{Q},\d\Omega\cap \Omega_{1})< \eta \ell(Q).
\end{equation*}
Let $Z\in \d\Omega\cap \Omega_{1}$ be such that
\[
\dist(Z,\d\Omega'\cap \Omega_{Q})< \eta \ell(Q).\]
See Figure \ref{f:damn}.a..

Let 
\[Z_{Q}\in \d\Omega'\cap \Omega_{Q}\cap B(Z,\eta\ell(Q))\] 
and set $B^{Q}=B(Z_{Q},\eta\ell(Q))$ (so we clearly have \eqn{rB} in this case). By \eqn{oqfar}, 
\[
B^Q
\subset B\ps{Z,2\eta \ell(Q)}\subset \Omega_{1}.\]

By \eqn{boundary} and \eqn{oqfar}
\begin{equation}\label{e:bz4eta}
B\ps{Z,4\eta\ell(Q)}\cap \d\Omega^{1}
=B\ps{Z,4\eta\ell(Q)}\cap (\Omega_{1}\cap \d\Omega)
\subseteq (\d\Omega^{1}\backslash (\Gamma\cap \Omega)).
\end{equation}
Let $Y\in B^{Q}\cap \d\Omega'$.  Since $B^{Q}\subset B(Z,2\eta\ell(Q))$, we know  $Y\in B\ps{Z,2\eta \ell(Q)}$, and since $Z\in \d\Omega\cap \Omega_{1}\subseteq \d\Omega^{1}$, we have
\[
\omega_{1}^{Y}(\d\Omega^{1}\backslash (\Gamma\cap \Omega))
\stackrel{\eqn{bz4eta}}{\geq} 
\omega_{1}^{Y}\ps{B\ps{Z,4\eta\ell(Q)}} \stackrel{\eqn{bourgain}}{\gec} 1\]
which proves \eqn{>B} in this case. Because each $B^{Q}$ is centered at a point in $\Omega_{Q}$, for $\eta$ small enough, we can guarantee  by definition of $\Omega_{Q}^{*}$ that \eqn{BinQ} holds for this case as well.
\begin{figure}[!ht]
\includegraphics[width=300pt]{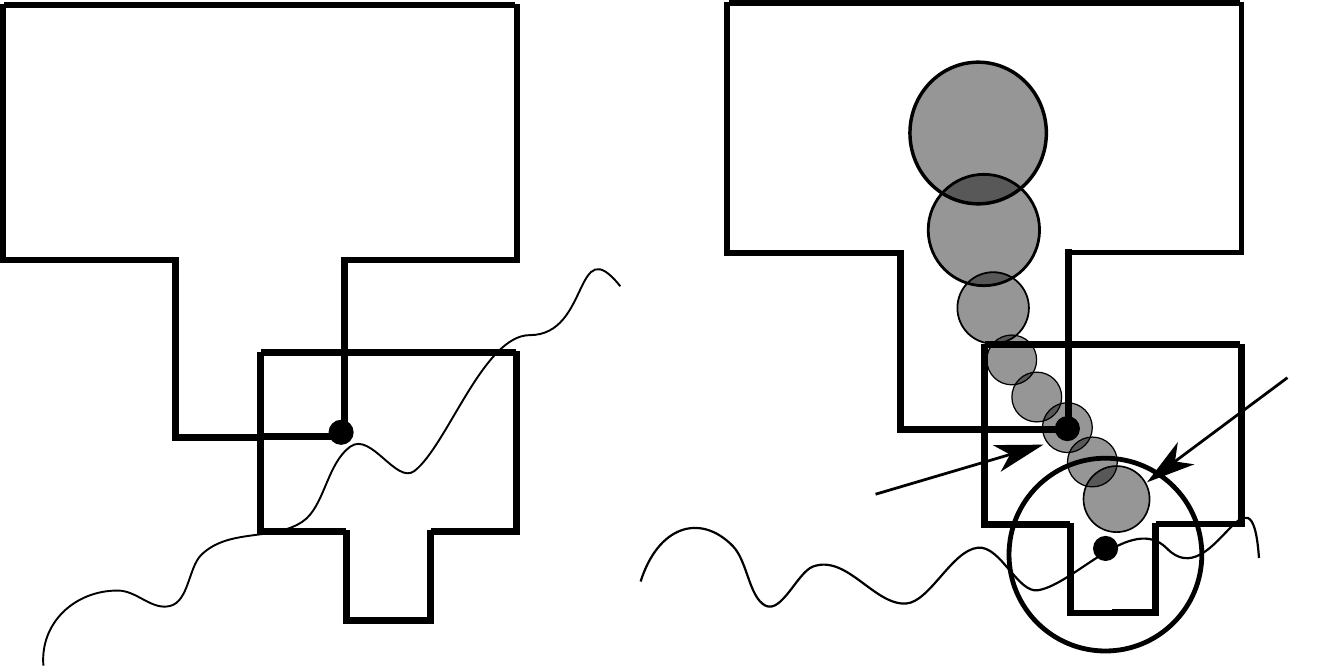}
\begin{picture}(0,0)(300,0)
\put(15,20){$\d\Omega$}
\put(145,35){$\d\Omega$}
\put(242,15){$Z$}
\put(77,52){$Z_{Q}$}
\put(50,125){$U_{Q}$}
\put(100,15){$U_{Q'}$}
\put(180,30){$Z_{Q}$}
\put(210,120){$B_{1}$}
\put(285,60){$B'$}
\put(75,0){a.}
\put(210,0){b.}
\end{picture}
\caption{}
\label{f:damn}
\end{figure}

{\bf Case 2:} Now suppose 

\begin{equation}\label{e:c2}
\dist(\d\Omega'\cap \Omega_{Q},\d\Omega\cap \Omega_{1})\geq  \eta \ell(Q).
\end{equation}

Note that by the properties of cubes $Q\in \F$, there is $Q'\in Q$ a child of $Q$ so that $U_{Q'}\cap \d\Omega\neq\emptyset$. By \eqn{oqfar} and \eqn{c2}, we can pick
\[
Z\in U_{Q'}\cap \d\Omega\subset \Omega_{Q}\cap \d\Omega\]
so that  
\begin{equation}\label{e:bzino1}
B(Z,\eta\ell(Q))\subset \Omega_{1}\backslash \Omega'.
\end{equation}
See Figure \ref{f:damn}.b.. Let $B'\subset U_{Q'}\cap B(Z,\eta\ell(Q)/2)$ be a corkscrew ball, so $r_{B'}\sim \eta \ell(Q)$.

Let $B_{1},...,B_{N}$ be a Harnack chain in $\Omega_{Q}$ from the center of $U_{Q}$ to the center of $B'$ so that $N\lec 1$. Let $Y_{1},...,Y_{N}$ denote their centers.  If we set $B_{N+1}=B(Z,\eta \ell(Q)/2)$, then $B_{N+1}$ is a corkscrew ball for $\Omega_{1}$ by \eqn{bzino1}, and so $B_{1},...,B_{N+1}$ form a Harnack chain in $\Omega_{1}$ with $B_{N+1}\cap \d\Omega\neq\emptyset$. Moreover, since $B_{1}$ contains a point in $U_{Q}\subset\Omega'$, and 
\[
B_{N+1}\subset B' \subset B(Z,\eta\ell(Q)) \stackrel{\eqn{bzino1}}{\subset} (\Omega')^{c},\] 
there is $j\leq N$ such that $B_{j}\cap \d\Omega'\neq\emptyset$, and so by applying  \Lemma{chain} to the chain $B_{j},...,B_{N+1}$ and using the fact that 
\[
B_{N+1}\cap \d\Omega^{1}=B_{N+1}\cap \d\Omega\subseteq \d\Omega^{1}\backslash (\Gamma\cap \Omega),\] 
we get
\begin{equation}\label{e:w1y}
\omega_{1}^{Y_{j}}(\d\Omega^{1}\backslash (\Gamma\cap \Omega))\gec 1.
\end{equation}
Since the $\{B_{i}\}_{i=1}^{N}$ form a Harnack chain in $\Omega_{Q}$, we know 
\[
\dist(B_{j},\Omega_{Q}^{c})\geq r_{B_{j}}/C\sim \ell(Q)\]
and so if we fix $Z_{Q}\in B_{j}\cap \d\Omega'\subseteq \d U_{Q}$, we have that, for $\eta$ small enough,
\[
B^{Q}:=B(Z_{Q},\eta \ell(Q))\subset \Omega_{Q}^{*}\subset \Omega_{1}\]
and so \eqn{BinQ} holds. Therefore, for $Y\in B^{Q}$, by Harnack's inequality,
\[
\omega_{1}^{Y}(\d\Omega^{1}\backslash (\Gamma\cap \Omega))\gec \omega_{1}^{Y_{j}}(\d\Omega^{1}\backslash (\Gamma\cap \Omega))\stackrel{\eqn{w1y}}{\gec} 1\]
which proves \eqn{>B} in this case. Again, \eqn{rB} holds by definition of $B^{Q}$.\\

For $B^{Q}$ chosen as in either case, we have by \eqn{BinQ} that
\[
B^{Q}\subset \Omega_{Q}^{*}:=\Omega_{\F_{Q}',Q}^{*}=U_{Q}^{*} \cup \bigcup_{R \in \F_{Q}} U_{R}^{*}.\]
Thus, since the $U_{Q}^{*}$ have bounded overlap in $\Omega_{1}$,
\[
\sum_{Q\in \F} \one_{B^{Q}} \leq \sum_{Q\in \bD}\one_{\Omega_{Q}^{*}}\lec \sum_{Q\in \bD} \one_{U_{Q}^{*}}\lec \one_{\Omega_{1}}\]
which proves \eqn{sumB}.

\end{proof}

\begin{lemma}\label{l:G}
Let $Q_0$ be the cube as chosen right after \eqn{close}, $\Omega'$ be as defined before \eqn{UinO}, and let $B^{Q}$ be as in \Lemma{3.2}. Define
\[
G= \ps{\d\Omega' \cap  Q_{0}}\cup \ps{\bigcup_{Q\in \F} B^{Q}\cap \d\Omega'}.\]

Then
\begin{equation}\label{e:atleastg}
\omega_{1}^{Y}(\d\Omega^{1}\backslash (\Gamma\cap \Omega))
\gec \omega_{\Omega'}^{Y}(G) \;\; \mbox{for $Y\in \Omega'$}.
\end{equation}
\end{lemma}

\begin{proof}
Let $\Phi$ be a superharmonic function in $\Omega^{1}$ so that for all $Z\in \d\Omega^{1}$,
\begin{equation}\label{e:liminf}
\liminf_{\Omega_{1}\ni Y\rightarrow Z} \Phi(y) \geq \one_{\d\Omega^{1}\backslash (\Gamma\cap \Omega)}(Z).\end{equation}
Then by the definition of harmonic measure using the Perron method,
\[
\Phi(Y)\geq \omega_{1}^{Y}(\d\Omega^{1}\backslash (\Gamma\cap \Omega)) \;\; \mbox{ for }\;\; Y\in \Omega^{1}.\]
Note that if $Q \in \F$ and $Z\in B^{Q}\cap \d\Omega'$, then by Lemma \ref{l:3.2}, and since $\omega_{1}$ is continuous at $Z$,
\[
\liminf_{\Omega'\ni Y\rightarrow Z} \Phi(Y) 
\geq \omega_{1}^{Z}(\d\Omega^{1}\backslash (\Gamma\cap \Omega))\gec 
1.\]

And if $Z\in  \d\Omega'\cap Q_{0}$, then because
\[
\d\Omega'\cap Q_{0}\subset \d\Omega'\cap \Gamma 
\subset \d\Omega^{1}\backslash (\Gamma\cap \Omega)
,\] 
we already have by \eqn{liminf} that
\[
\liminf_{\Omega'\ni Y\rightarrow Z} \Phi(Y) \gec 1,\]
and thus $\Phi$ is also an upper function for $G$ in $\Omega'$, hence
\[
\Phi(Y) \geq \omega_{\Omega'}^{Y}(G) \;\; \mbox{ for }Y\in \Omega'.\]
Infimizing over all upper functions $\Phi$ for $\one_{\d\Omega^{1}\backslash (\Gamma\cap \Omega)}$ completes the proof.
\end{proof}

%
%

%

\begin{lemma}
\label{l:lQ}
Let $\F$ and $Q_0$ be as in \Lemma{G} and let and $B'_{Q_{0}}$ be defined as in \Lemma{HM3} associated to $Q_0$. Then we have
\begin{equation}\label{e:content1}
\sum_{Q\in \F} \ell(Q)^{d}
\gec \cH^{d}_{\infty}(\Omega^{c} \cap \Omega_{1}\cap B_{Q_{0}}') 
.\end{equation}
\end{lemma}

\begin{proof}
By \Lemma{HM3}
\[
B_{Q_{0}}' \cap \Omega_{1} \subset T_{Q_{0}}\cap \Omega_{1}.\]
By definition, if $Q\in \bD_{\F,Q_{0}}$, then $U_{Q}\cap \Omega^{c}=\emptyset$. Hence,
\begin{align*}
B_{Q_{0}}' \cap\Omega_{1} \cap \Omega^{c}
& \subset T_{Q_{0}}\cap \Omega_{1}\cap \Omega^{c}\\
&=\ps{\bigcup_{Q\in \F}T_{Q}
\cup \bigcup_{Q\in \bD_{\F,Q_{0}}}U_{Q}} \cap  \Omega_{1} \cap \Omega^{c}\\
& = \bigcup_{Q\in \F}T_{Q}\cap  \Omega_{1} \cap \Omega^{c}.\end{align*}
Thus,
\begin{equation*}
\sum_{Q\in \F} \ell(Q)^{d}
\gec \sum_{Q\in \F} (\diam {T_{Q}})^{d}
\gec \cH^{d}_{\infty}(\Omega^{c} \cap \Omega_{1}\cap B_{Q_{0}}').
\end{equation*}

\end{proof}
%
%

\begin{lemma}
\label{l:Hcbdd}
Let $\F$ and $B'_{Q_{0}}$ be as in \Lemma{lQ}. Then one has
\begin{align}\label{e:woops}
\cH^{d}(\d\Omega'\cap \d\Omega_{1}\cap B_{Q_{0}}')+ \sum_{Q\in \F} \ell(Q)^{d}
& \gec \cH^{d}\ps{ \d\Omega'\cap \d\Omega_{1} \cap B_{Q_{0}}' \cup \bigcup_{Q\in \F}Q} \notag \\
& \geq \cH^{d}(\d\Omega\cap \d\Omega_{1}\cap B_{Q_{0}}').\end{align}
\end{lemma}

\begin{proof}
By \Lemma{Christ} (i), {$\cH^{d}$} almost every $x\in Q_{0}$ is contained in either a cube from $\F'$ (and so in a cube from $\T\cup \F$) or every cube in $Q_{0}$ containing $x$ is in $\bD_{\F', Q_0}$. Hence, for almost every $x\in \d\Omega\cap \d\Omega_{1}$, $x$ cannot be in a cube from $\T$ by definition (since $x$ is in a cube from $T$ only if that cube is in $Q_{0}\backslash \d\Omega$), so it must be in a cube from $\F$ or infinitely many cubes from $\bD_{\F', Q_0}$, and in the latter case, we must have that $x\in \d\Omega'\cap \d\Omega_{1}$ by construction. Thus, we have shown that { $\cH^{d}$} almost every point in $\d\Omega\cap \d\Omega_{1}$ is in ${ \d\Omega'\cap \d\Omega_{1} \cup \bigcup_{Q\in \F}Q}$, which proves the second inequality. The first inequality follows since
\[
\cH^{d}\ps{\bigcup_{Q\in \F}Q}
\leq \sum_{Q\in \F} \cH^{d}(Q)
\sim \sum_{Q\in \F} \ell(Q)^{d}.\]
\end{proof}

\begin{lemma}
Let $\F$ and $B'_{Q_{0}}$ be as in \Lemma{Hcbdd}. Then we have 
\begin{equation}\label{e:long}
\cH^{d}_{\infty}(\Omega^{c} \cap \Omega_{1}\cap B_{Q_{0}}') \gec \cH^{d}_{\infty}((\Omega^{c})^{\circ}\cap \d\Omega_{1}\cap B_{Q_{0}}')  .
\end{equation}
\end{lemma}

\begin{proof}
Let $Q_{j}$ be maximal cubes in $(\Omega^{c})^{\circ}\cap \d\Omega_{1}\cap B_{Q_{0}}'$ for which $\cnj{B_{Q_{j}}}\subset  B_{Q_{0}}'$. Then by \Lemma{Christ} (i), they cover almost all of $(\Omega^{c})^{\circ}\cap \d\Omega_{1}\cap B_{Q_{0}}'$. It is not hard to show that $\d\Omega_{1}\cup \bigcup \d B_{Q_{j}}$ is also Ahlfors regular and $\cH^{d}(\d B_{Q_{j}}\cap \Omega_{1})\sim r_{Q_{j}}^{d}$ by the Harnack chain and corkscrew conditions for $\Omega_{1}$. Also recall from \eqn{bqbrempty} that the $B_{Q_{j}}$ are disjoint. Hence,
\begin{align*}
\cH^{d}_{\infty}((\Omega^{c})^{\circ}\cap \d\Omega_{1}\cap B_{Q_{0}}')
& \stackrel{\eqn{hhfin}}{\sim} \cH^{d}((\Omega^{c})^{\circ}\cap \d\Omega_{1}\cap B_{Q_{0}}')  
 =\sum \cH^{d}(Q_{j}) \\
 & \sim \sum r_{Q_{j}}^{d}
 \sim \sum \cH^{d}(\d B_{Q_{j}}\cap \Omega_{1})\notag \\
& \stackrel{\eqn{bqbrempty}}{=}  \cH^{d}\ps{\bigcup \d B_{Q_{j}}\cap \Omega_{1}}
 \stackrel{\eqn{hhfin}}{\sim}\cH^{d}_{\infty}\ps{\bigcup \d B_{Q_{j}}\cap \Omega_{1}}\notag \\
& \leq \cH^{d}_{\infty}((\Omega^{c})^{\circ}\cap \Omega_{1}\cap B_{Q_{0}}')
 \leq \cH^{d}_{\infty}(\Omega^{c} \cap \Omega_{1}\cap B_{Q_{0}}').
\end{align*}
\end{proof}

\begin{lemma}
Let $G$ be as in \Lemma{G}. Then
\[
\cH^{d}(G)\gec r^{d}.
\]
\end{lemma}

\begin{proof}
First, we record a few facts. By \Lemma{HM1}, $\d\Omega'$ is Ahlfors regular, so 
\begin{equation}\label{e:B^Qsimr}
\cH^{d}(B^{Q}\cap \d\Omega')\sim r_{B^{Q}}^{d} \stackrel{\eqn{rB}}{\sim} \ell(Q)^{d}.
\end{equation}
Since
\begin{equation}\label{e:dcomp}
\Omega^{c}\cap \cnj{\Omega_{1}}
=(\d\Omega^{c}\cap \d\Omega_{1})\cup (\Omega^{c}\cap \Omega_{1})\cup ((\Omega^{c})^{\circ}\cap \d\Omega_{1}).
\end{equation}
and because the $B^{Q}$ have bounded overlap by \eqn{sumB}, we get
\begin{align*}
\cH^{d}(G)
& \gec \cH^{d}(\d\Omega'\cap Q_{0})+ \sum_{Q\in \F} \cH^{d}(B^{Q}\cap \d\Omega')\\
& \stackrel{\eqn{B^Qsimr}}{\gec}   \cH^{d}(\d\Omega'\cap \d\Omega_{1}\cap B_{Q_{0}}')+ \sum_{Q\in \F} \ell(Q)^{d}\\
& \geq  \frac{1}{2}\ps{ \cH^{d}(\d\Omega'\cap\Omega_{1}\cap B_{Q_{0}}')+ \sum_{Q\in \F} \ell(Q)^{d}} 
 + \frac{1}{2} \sum_{Q\in \F} \ell(Q)^{d} \\
& \stackrel{ \eqn{woops}}{\gec} \cH^{d}(\d\Omega\cap \d\Omega_{1}\cap B_{Q_{0}}')+  \frac{1}{2} \sum_{Q\in \F} \ell(Q)^{d}\\
&  \stackrel{\eqn{content1}}{\gec} 
\cH^{d}(\d\Omega\cap \d\Omega_{1}\cap B_{Q_{0}}')+
\cH^{d}_{\infty}(\Omega^{c}\cap  \Omega_{1} \cap B_{Q_{0}}')\\
& \stackrel{\eqn{long}}{\gec} \cH_{\infty}^{d}(\d\Omega^{c}\cap  \d\Omega_{1}\cap B_{Q_{0}}')+ \cH^{d}_{\infty}(\Omega^{c}\cap   \Omega_{1}\cap B_{Q_{0}}' )\\ 
& \qquad + \cH^{d}_{\infty}((\Omega^{c})^{\circ}\cap \d\Omega_{1}\cap B_{Q_{0}}')
\\
& \stackrel{\eqn{dcomp}}{\geq } \cH^{d}_{\infty}(\Omega^{c}\cap \cnj{\Omega_{1}}\cap B_{Q_{0}}')
\stackrel{\eqn{om1content}}{\gec} r^{d} .
\end{align*}
\end{proof}

%

Pick a ball $B$ centered on $G$ of radius $c\frac{\diam \Omega'}{2C_{DJ}}$ (where $c$ is the interior corkscrew constant for $\Omega'$ and $C_{DJ}$ depends on the NTA and Ahlfors regularity constants for $\Omega'$) such that
\[
\cH^{d}(B\cap G)\gec \cH^{d}(G)\gec r^{d} .\]

Since $\Omega'$ is Ahlfors regular, $r^{d}\sim \cH^{d}(B\cap \d\Omega')$, and so $\cH^{d}(G\cap B )/\cH^{d}(B\cap \d\Omega')\gec 1$. Let $X_{B}$ be a corkscrew point in $B\cap \Omega'$. Pick $Z_{0}\in \Omega'$ so that
\[
B(Z_{0},c \diam \Omega') \subset \Omega'\]
where again $c$ is the interior corkscrew constant for $\Omega'$.

By our choice of $B$, $Z_{0}\not\in C_{DJ}B$, and so \Theorem{DJ} and the Harnack chain condition imply 
\[
\omega_{\Omega'}^{Z_{0}}(G\cap B)\gec \omega_{\Omega'}^{Z_{0}}(B)\sim \omega_{\Omega'}^{X_{B}}(B)\gec 1.\]

Finally, let $B_{1},...,B_{N}$ be a Harnack chain in $\Omega_{1}$ from $Y^{1}$ to $Z_{0}$ and let $Y_{j}$ denote the center of $B_{j}$ where $Y^{1}$ is as in \eqn{balli}. Then $N\lec 1$ and so $r_{B_{i}}\sim r_{B_{1}}$ for $1\leq i\leq N$. Thus, by \eqn{close}, for $\ve>0$ small enough depending on the NTA constants for $\Omega_{1}$, we can guarantee that
\[
\dist(B_{i},\d\Omega)\gec r_{B_{i}} \]
and since $B_{i}$ is a Harnack chain in $\Omega_{1}$, we already have 
\[
\dist(B_{i},\d\Omega_{1})\gec r_{B_{i}},\]
so in particular,
\[
\dist(B_{i},\d\Omega^{1})\gec r_{B_{i}}\]
Thus, using Harnack's inequality and \Lemma{G}, we get that for all $Y\in B_{1}$,
\begin{align*}
\omega_{\Omega^{1}}^{Y}(\d\Omega^{1}\backslash (\Gamma\cap \d\Omega))
& \gec \omega_{\Omega^{1}}^{Y_{B_{1}}}(\d\Omega^{1}\backslash (\Gamma\cap \d\Omega))\\
&\gec \omega_{\Omega^{1}}^{Z_{0}}(\d\Omega^{1}\backslash (\Gamma\cap \d\Omega))\\
& \stackrel{\eqn{atleastg}}{\gec} \omega_{\Omega'}^{Z_{0}}(G)
\geq  \omega_{\Omega'}^{Z_{0}}(G\cap B)
\gec 1.
\end{align*}

This proves \eqn{eta2}, and thus completes the proof of \TheoremI.

\section{The Proof of \TheoremII}
\label{s:cones}

\TheoremII will follow quickly from Lemma \hyperref[l:lemI]{I } and the following lemma.

\begin{lemii}\label{l:lemII}
Let $\Omega\subset \bR^{d+1}$ be a bounded domain with big boundary and assume $\Omega$ is contained in a domain $\Omega_{0}$ whose boundary is a Lipschitz graph. If $\omega_{\Omega}^{X_{0}}(\d\Omega_{0}\cap \d\Omega)>0$ for some $X_{0}\in \Omega$, then $\omega_{\Omega}$-almost every point in $\d\Omega_{0}\cap \d\Omega$ is a cone point for $\Omega$.
\end{lemii}

\begin{proof}[Proof of \TheoremII]
Suppose there is $F\subset \Gamma\cap \d\Omega$ with $\omega_{\Omega}^{X_{0}}(F)>0$ but no point in $F$ is a cone point for $\Omega$. By \Lemma{reduction2}, we may find a connected open set $\hat{\Omega}\subset \Omega$ bounded with big boundary such that $\omega_{\hat{\Omega}}^{\hat{X}}(F)>0$ for some $\hat{X}\in \hat{\Omega}$ in the same component of $\Omega$ as $X_{0}$.

Let $\Omega_{1}$ and $\Omega_{2}$ be the components of $\Gamma^{c}$. Since they are both NTA domains and $\Gamma$ is Ahlfors regular (by virtue of being a Lipschitz graph), Lemma \hyperref[l:lemI]{I } implies there is $i\in \{1,2\}$ and $X_{i}\in \hat{\Omega}\cap \Omega_{i}$ so that 
\begin{equation*}
\omega_{\hat{\Omega}\cap \Omega_{i}}^{X_{i}}(F)>0.
\end{equation*}

Now we can apply Lemma \hyperref[l:lemII]{II}--where we have $\hat{\Omega}\cap \Omega_{i}$ in place of $\Omega$, $F$ in place of $E$, and $\Omega_{i}$ in place of $\Omega_{0}$---to conclude that if $F'\subset F$ are the cone points for $\hat{\Omega}\cap \Omega_{i}$, then $\omega_{\hat{\Omega}\cap \Omega_{i}}^{X_{i}}(F')>0$. By containment, we also know that they are also cone points for ${\Omega}$. By \hyperref[l:CP]{Carleman's Principle},
\[
0<\omega_{\hat{\Omega}\cap \Omega_{i}}^{X_{i}}(F')
\leq \omega^{X_{i}}_{\hat{\Omega}}(F') \leq \omega^{X_{i}}_{{\Omega}}(F').\]
Since $X_{i}$ is in the same component of $\Omega$ as $X$, this also implies $\omega_{\Omega}^{X_{0}}(F')>0$, and thus the set of cone points for $\Omega$ has positive $\omega_{\Omega}^{X_{0}}$-measure, which is a contradiction. \\

\end{proof}

The rest of this section is devoted to proving Lemma \hyperref[l:lemII]{II}, but before we do so, we recall some background on the tangent measures of David Preiss \cite{Pr87}.

For $x,y\in\bR^{d+1}$ and $r>0$, define
$$T_{x,r}(y) := \frac{y-x}{r}.$$
Note that $T_{x,r}(B(x,r))=B(0,1)$. Given a Radon measure $\mu$, the notation $T_{x,r}[\mu]$ stands for the image measure of $\mu$ by $T_{x,r}$. That is,
$$T_{x,r}[\mu](A) = \mu(rA+x),\qquad A\subset\bR^{d+1}.$$

\begin{definition}  
Let $\mu$ be a Radon measure in $\bR^{d+1}$. We say that $\nu$ is a {\it tangent measure} of $\mu$ at a point $x\in\bR^{d+1}$, denoted as $\nu\in \Tan(\mu,x)$, if
$\nu$ is a non-zero Radon measure on $\bR^{d+1}$ and there are sequences $\{r_i\}_i$ and $\{c_i\}_i$ of positive numbers, with $r_i\to0$, so that $c_i\,T_{x,r_{i}}[\mu]$ converges weakly to $\nu$ as $i\to\infty$.
\end{definition}

\begin{lemma} \cite[Theorem 14.3]{Mattila} Let $\mu$ be a Radon measure on $\bR^{d+1}$. If $x\in \bR^{d+1}$ and 
\begin{equation*}
\limsup_{r\rightarrow 0} \frac{\mu(B(x,2r))}{\mu(B(x,r))}<\infty.
\end{equation*}
then every sequence $\{r_{i}\}_{i}$ with $r_{i}\downarrow 0$ contains a subsequence (denoted $\{r_i\}_i$ again) such that the measures $T_{x,r_{i}}[\mu]/\mu(B(x,r_{i}))$ converges to a measure $\nu\in \Tan(\mu,x)$. 
\label{l:tanexist}
\end{lemma}

\begin{lemma} \label{l:same-tangents}
\cite[Lemma 14.5]{Mattila}
Let $\mu$ be a Radon measure on $\bR^{d+1}$ and $A$ a  measurable set. Suppose $x\in \supp \mu$ is a point of density for $A$, meaning
\[\lim_{r\rightarrow 0} \frac{\mu(B(x,r)\backslash A)}{\mu(B(x,r))}=0.\]
If $c_{i} T_{x,r_{i}}[\mu] \rightarrow \nu\in \Tan(\mu,x)$, then so does $c_{i} T_{x,r_{i}}[\mu]|_{A}$. In particular, this holds for $\mu$ almost every $x\in A$. 
\end{lemma}

The above lemma is not stated as such in \cite{Mattila}, but it follows by an inspection of the proof (in particular the last two lines).

\begin{lemma}[{\cite[Lemma 14.6]{Mattila}}]\label{l:munu}
Let $\mu,\nu$ be Radon measures such that $\mu=g\nu$ for some non-negative locally $\nu$ integrable function $g$ in $\bR^{d+1}$. Then for $\nu$-almost every $x\in \bR^{d+1}$, $\Tan(\mu,x)=\Tan(\nu,x)$. In particular, if $\nu\ll \mu$, then for $\nu$-almost every $x\in \bR^{d+1}$, $\Tan(\mu,x)=\Tan(\nu,x)$.
\end{lemma}

\begin{definition} \label{d:defdelta}
A domain $\Omega\subsetneq \bR^{d+1}$ is {\it $\Delta$-regular} if there is 
$R>0$ so that 
\begin{equation*}
 \sup_{x \in\d\Omega} \,\sup_{X\in \d B(x,r/2)\cap\Omega} \omega_{B\cap \Omega}^{X}(\d B(x,r)\cap \Omega) <1 \mbox{ for }r\in (0,R).
 \end{equation*}
 \end{definition}

By \Lemma{bourgain}, any domain satisfying \eqn{dcontent} is $\Delta$-regular. \\

Here we recall a truncated version of a lemma from \cite{AMT16}. It is a generalization of similar results that first appeared in the works of Kenig, Preiss, and Toro, who first noted the connections between tangent measure techniques and studying harmonic measure (see \cite{KPT09,KT99,KT06}).

\begin{lemma}[{\cite[Lemma 5.9]{AMT16}}]\label{l:blowup}
Let $\Omega\subset \bR^{d+1}$ be a $\Delta$-regular domain. Let $\omega=\omega_{\Omega}^{X_{0}}$ for some $X_{0}\in \Omega$. Let $x\in \d\Omega$ and
$\omega_\infty\in \Tan(\omega,x)$, with
$\{c_{j}\}_{j}$ with $c_{j}\geq 0$, and $\{r_{j}\}_{j}$ with $r_{j}\rightarrow 0$ such that $\omega_{j}=c_{j}T_{x,r_{j}}[\omega]\rightarrow \omega_{\infty}$. Let $\Omega_{j}=T_{x,r_{j}}(\Omega)$. Then there is a subsequence and a closed set $\Sigma\subset \bR^{d+1}$ such that 
\begin{enumerate}[(a)]
\item $\d\Omega_{j}\cap K\rightarrow \Sigma\cap K$ in the Hausdorff metric for any compact set $K$.
\item $\Sigma^{c}=\Omega_{\infty}\cup \mbox{ext}(\Omega_{\infty})$ where $\Omega_{\infty}$ is a nonempty open set and $\mbox{ext}(\Omega_{\infty})$ is also open but possibly empty. Further, they
satisfy that for any ball $B$ with $\cnj{B}\subset \Omega_{\infty}$, a neighborhood of $\cnj{B}$ is contained in $\Omega_{j}$ for all $j$ large enough.
\item $\supp \omega_{\infty}\subset \Sigma$. 
\item Let $u(X)=G_{\Omega}(X,X_{0})$ on $\Omega$ and $u(X)=0$ on $\Omega^{c}$, where $G_{\Omega}$ is the Green function for $\Omega$. Set
\[u_{j}(X)=c_{j}\,u(Xr_{j}+x)\,r_{j}^{d-1}.\]
Then $u_{j}$ converges uniformly on compact subsets of $\bR^{d+1}$ to a nonzero function $u_{\infty}$ that is harmonic on $\Omega_{\infty}$ such that 
for any smooth compactly supported function $\phi$,
 \begin{equation}\label{e:ibp}
 \int_{\d\Omega} \phi \,d\omega_{\infty}=\int_{\Omega} \Delta\phi \,u_{\infty}\,dX
 \end{equation}
 \end{enumerate}
 \end{lemma}
The above is a truncated verison of the original theroem. Moreover, the original theroem was stated for $d>1$, bu the part that we have cited holds for $d=1$ as well. Referring to their paper, the only place where the assumption that $d>1$ was used was in order to use \cite[Lemma 4.1 (4.7)]{AMT16}, but this inequality holds also for $d=1$ by \cite[Lemma 3.2]{AH08} and the maximum principle as in the proof of \cite[Lemma 3.5]{AH08}.  We refer the reader to \cite{AMT16} for the complete details.
 
 \begin{lemma} Under the assumptions of \Lemma{blowup}, 
 \begin{equation}\label{e:uw}
 \d\{u_{\infty}>0\}=\supp \omega_{\infty}.
 \end{equation}
 \end{lemma}
 
 \begin{proof}
 Let $x\in \supp \omega_{\infty}$ and suppose there is a ball $B\subset \{u_{\infty}>0\}$ containing $x$. Let $\phi$ be any smooth function $\phi$ supported in $B$ so that $\phi(x)>0$. Then by Green's theorem,
 \[
0< \int \phi d\omega_{\infty} \stackrel{\eqn{ibp}}{=} \int \Delta \phi u_{\infty}=0
\]
which is  a contradiction. We obtain a similar contradiction more easily if there is a ball $B\subset \{u_{\infty}=0\}$ containing $x$. Thus, all balls containing $x$ must intersect both $\{u_{\infty}=0\}$ and $\{u_{\infty}>0\}$, hence $x\in \d\{u_{\infty}>0\}$, which implies $\supp \omega_{\infty}\subset \d\{u_{\infty}>0\}$. 

Now let $x\in \d\{u_{\infty}>0\}$ and suppose there is $B\subset (\supp \omega_{\infty})^{c}$ containing $x$. Then for any smooth function $\phi$ supported in $B$, we have
\[\int \Delta \phi u_{\infty} \stackrel{\eqn{ibp}}{=} \int \phi d\omega_{\infty}=0.\]
Thus, $u_{\infty}$ is harmonic in $B$, and since it is nonnegative and continuous up to the boundary, it achieves its minimum only at some point in $\d B$ by the strong maximum principle, hence $u_{\infty}>0$ in $B$. However, as $x\in \d\{u_{\infty}>0\}$, $B\cap \{u_{\infty}=0\}\neq\emptyset$, and so $u_{\infty}=0$ somewhere in $B$ which is a contradiction. Thus, every ball containing $x$ intersects $\supp \omega_{\infty}$, which implies $x\in \supp \omega_{\infty}$ since $\supp \omega_{\infty}$ is closed. Hence, $\d\{u_{\infty}>0\}\subset \supp \omega_{\infty}$, and we are done.
 \end{proof}

We now proceed with the proof of \TheoremII. By \TheoremI, $\omega_{\Omega}\ll \cH^{d}$ on $E:=\d\Omega_{0}\cap \d\Omega$. Let $E'\subset E$ be such that $\omega_{\Omega}(E\backslash E')=0$ and 
\[
\omega_{\Omega}|_{E'}\ll \cH^{d}|_{E'} \ll \omega_{\Omega}|_{E'}.\]
\begin{lemma}\label{l:E''}
There is $E''\subset E'$ so that $\omega_{\Omega}(E\backslash E'')=0$ and for all $x\in E''$ there is a $d$-plane $V_{x}$ so that for every sequence $\{r_{j}\}_j$ with $r_{j}\downarrow 0$, we may pass to a subsequence so that $T_{x,r_{j}}[\omega_{\Omega}]/\omega_{\Omega}(B(x,r_{j}))$ converges weakly to a multiple of $\cH^{d}|_{V_{x}}$.
\end{lemma}

\begin{proof}
Because $\d\Omega_{0}$ is the graph of an almost everywhere differentiable function, for almost every $x\in \Gamma$, there is a $d$-plane $V_{x}$ so that 
 \begin{equation}\label{e:tantan}
  \Tan(\cH^{d}|_{\d\Omega_{0}},x)=\{c\cH^{d}|_{V_{x}}: c>0\}.
 \end{equation}

Note that $d\omega_{\Omega}|_{E'}=gd\cH^{d}$ for some measurable function $g$ that is positive and finite almost everywhere on $E'$ and zero everywhere else, so by the Lebesgue density theorem, for almost every $x\in E'$,
\[
\lim_{r\rightarrow 0}
\frac{\omega_{\Omega}(E'\cap B(x,r))}{\cH^{d}(B(x,r))}=g(x)\in (0,1)\]
and so
\begin{multline*}
\limsup_{r\rightarrow 0} \frac{\omega_{\Omega}(E'\cap B(x,2r))}{\omega_{\Omega}(E'\cap B(x,r))}\\
=\limsup_{r\rightarrow 0} 
\frac{\omega_{\Omega}(E'\cap B(x,2r))}{\cH^{d}(B(x,2r))}
\frac{\cH^{d}(B(x,2r))}{\cH^{d}(B(x,r))}
\frac{\cH^{d}(B(x,r))}{\omega_{\Omega}(E'\cap B(x,r))}
<\infty.\end{multline*}
Hence, \Lemma{tanexist}, \eqn{tantan}, and our choice of $E'$  imply that for any sequence $\{r_j\}_j$ with $r_{j}\downarrow 0$ we may pass to a subsequence so that $T_{x,r_{j}}[\omega_{\Omega}]/\omega_{\Omega}(B(x,r_{j}))$ converges weakly to a multiple of $\cH^{d}|_{V_{x}}$. We now let $E''$ be the set of $x\in E'$ for which this occurs, which is almost all of $E'$.
\end{proof}

Now we will show that each $x\in E''$ is a cone point. Fix $x\in E''$ and  let $v_{x}\in \bS^{d}$ be the vector normal to $V_{x}$ such that 
 \begin{equation}\label{e:x+tv}
 \{x+tv_{x}:t>0\}\subset \Omega_{0}^{c}.\end{equation}
 Set 
 \[
 H_{x}^{\pm}=\{y\in \bR^{d}: \pm y\cdot v_{x}>0\}\]
 so that $V_{x}^{c}= H_{x}^{+}\cup H_{x}^{-}$.
 
Let
\[
C'(x,r)=C(x,-v_{x},1/2,r)\backslash C(x,-v_{x},1/2,r/2)
\]
where $C(\cdot, \cdot, \cdot, \cdot)$ is defined as above \TheoremII. Suppose there was $r_{j}\downarrow 0$ so that for all $j$ we could find 
\[
X_{j}\in C'(x,r_{j})\cap \Omega^{c}\neq \emptyset.\]
By \Lemma{E''}, we may pass to a subsequence so that
\begin{equation*}
\omega_{j}= T_{x,r_{j}}[\omega_{\Omega}]/\omega_{\Omega}(B(x,r))\rightarrow \omega_{\infty} \neq 0
\end{equation*}
and 
\begin{equation}\label{e:supwv}
\supp \omega_{\infty}=V_{x}.
\end{equation}
Pass to a further subsequence so that the conclusions of \Lemma{blowup} hold. By \eqn{x+tv}, $u=0$ on $ \{x+tv_{x}:t>0\}$, and thus we know $u_{j}=0$ on $\{tv_{x}:t>0\}$. Since $u_{j}\rightarrow u_{\infty}$ uniformly on compact subsets, we also know $u_{\infty}=0$ on $\{tv_{x}:t>0\}\subset H_{x}^{+}$. If $u_{\infty}(X)>0$ for some $X\in H_{x}^{+}$, then the line segment between $X$ and $v_{x}$ is contained in $H_{x}^{+}$ and intersects with
\[
\d \{u_{\infty}>0\}\stackrel{\eqn{uw}}{=}\supp \omega_{\infty}\stackrel{\eqn{supwv}}{=}V_{x}\subset (H_{x}^{+})^{c},\]
which is a contradiction. Thus, 
\begin{equation}\label{e:ufin0}
\mbox{$u_{\infty}=0$ on all of $H_{x}^{+}$}.
\end{equation}

Let \[
Y_{j}=T_{x,r_{j}}(X_{j})
\in C'(0,1).\]
We may pass to a further subsequence so that 
\[
Y_{j}\rightarrow Y\in \cnj{C'(0,1)} \subset H_{x}^{-}.\]
Since $X_{j}\in \Omega^{c}$, we know $u(X_{j})=0$ and hence $u_{j}(Y_{j})=0$ as well. Since $u_{j}\rightarrow	u_{\infty}$ uniformly on compact sets, we know $u_{\infty}(Y)=0$. Because $\omega_{\infty}\neq 0$, $u_{\infty}$ is not identically zero, but  \eqn{ufin0} implies there is $W\in H_{x}^{-}$ so that $u_{\infty}(W)>0$. But then the line segment between $Y$ and $W$ is contained in $H_{x}^{-}$ and intersects 
\[
\d \{u_{\infty}>0\}\stackrel{\eqn{uw}}{=}\supp \omega_{\infty}\stackrel{\eqn{supwv}}{=}V_{x}\subset (H_{x}^{-})^{c},\]
which leads us to another contradiction. Therefore, we now know that for $r>0$ sufficiently small, $C'(x,r)\cap \Omega^{c}=\emptyset$, which implies that $C(x,r)\subset \Omega$ for $r>0$ small enough. Thus, $x$ is a cone point.\\

\section{The Proof of \TheoremIII}\label{sec:III}

Now we prove \TheoremIII. \\

{\bf (1)$\Rightarrow$(2):} This is just \Theorem{7per}. \\

{\bf (2)$\Rightarrow$(1):} Suppose $E$ is covered by countably many Lipschitz graphs up to harmonic measure zero. Since each graph is the boundary of a two-sided NTA domain with Ahlfors regular boundary, by \TheoremI we get $\omega_{\Omega}\ll \cH^{d}$  on each Lipschitz graph, and thus on all of $E$. \\

{\bf (3)$\Rightarrow$(2):} It is well known that the set of cone points can be covered by countably many Lipschitz graphs, see for example \cite[Lemma 15.13]{Mattila}. \\

{\bf (2)$\Rightarrow$(3):} Assume $E\subset \d\Omega$ can be covered up to $\omega_{\Omega}^{X_{0}}$-measure zero by countably many Lipschitz graphs $\Gamma_{i}$. Then $\omega_{\Omega}$-almost every point in $\Gamma_{i}\cap \d\Omega$ is a cone point by \TheoremII, and thus $\omega_{\Omega}$-almost every point in $E$ is a cone point. \\

Finally, we prove \eqn{whw}. Firstly, we have already shown that $\omega_{\Omega}\ll \cH^{d}$ on $F$ since $F$ is contained in a countable union of Lipschitz graphs, so we need only show $\cH^{d}\ll \omega_{\Omega}$ on $F$.

It is not hard to show that there are domains $\Omega_{i}\subset \Omega$ whose boundaries are a finite union of Lipschitz graphs such that 
\begin{equation}\label{e:FinO}
F\subset \bigcup \d\Omega_{i}.\end{equation}
By \hyperref[l:CP]{Carleman's Principle} and \Theorem{DJ}, since Lipschitz domains are NTA domains with Ahlfors regular boundaries, 
\[
\cH^{d}|_{\d\Omega_{i}\cap F}
\ll \omega_{\Omega_{i}}|_{\d\Omega_{i}\cap F}
\ll \omega_{\Omega}|_{\d\Omega_{i}\cap F}\]
and thus by \eqn{FinO}, $\cH^{d}\ll \omega_{\Omega}$ on $F$.

%
%
%
%
%
%
%
%

\appendix

\section{Generalizing David-Jerison with Kenig-Pipher}

The goal of this section is to sketch a proof of \Theorem{DJ}. In the rest of this section we will explore an example of a class of elliptic operators with the \hyperref[d:KP]{KP-condition} is satisfied, any sub-NTA domain with Ahlfors regular boundary also has that elliptic measure is $A_{\infty}$-equivalent to $\cH^{d}$.


For a domain $\Omega \subset \R^{d+1}$,  $Z\in \Omega$ and a uniformly elliptic matrix $\mathcal{A}$, we recall

\[
\ve_{\Omega}^{\mathcal{L}}(Z) :=\sup \{ \dist(X, \partial \Omega) |\nabla \mathcal{A}(X)|^2 : X \in B(Z, \dist(Z, \partial \Omega)/2)\},\]
where we abuse notation by setting
\[
|\nabla \mathcal{A}(X)|:=\max_{1\leq i,j\leq d+1}|\nabla a_{ij}(X)|.\]

\begin{lemma}\label{l:Car2to1}
Let  $\Omega_1 \subset \Omega_2$ and let the matrix $\mathcal{A}$ be uniformly elliptic in $\Omega_2$ so that its distributional derivatives satisfy
\begin{equation}\label{e:KP-Car}
\frac{1}{r^n} \int_{B(x,r) \cap \Omega_2}  \varepsilon_{\Omega_2}^{\mathcal{L}}(Z)\,dZ \leq C,
\end{equation}
for any $x \in \partial \Omega_2$ and $r \in (0, \diam \Omega_2)$. Here $dZ$ stand for the $(d+1)$-dimensional Lebesgue measure. Then \eqref{e:KP-Car} also holds with $\Omega_{1}$ in place of $\Omega_{2}$. 
\end{lemma}

This is sketched in \cite[Section 3.2]{ABHM15}, but we provide some details here.

\begin{proof}
Let us first assume that $\xi \in \partial \Omega_1 \cap \Omega_2$ and $r \leq  \frac{\dist(\xi, \partial \Omega_2)}{80}$. Then, if $Z\in B(\xi,r) \cap \Omega$, we have that for any $Y \in B(\xi,  \dist(\xi, \partial \Omega_2)/4)$,

\begin{align*}
\varepsilon_{\Omega_1}^{\mathcal{L}}(Z) &\leq r \sup_{X \in B(\xi, 2r)} |\nabla \mathcal{A}(X)|^2 \\
&\leq  \sup\{ \dist(X, \partial \Omega_2) |\nabla \mathcal{A}(X)|^2 :X \in B(Y, \dist(Y, \partial \Omega_2)/2)\}\\
&= \varepsilon_{\Omega_2}^{\mathcal{L}}(Y),
\end{align*}
 Using this we get that

\begin{align*}
\frac{1}{r^{d}}& \int_{B(\xi,r) \cap \Omega_1}  \varepsilon_{\Omega_1}^{\mathcal{L}}(Z)\,d Z 
\lesssim r\inf_{Y\in B(\xi,\d\Omega_{2})/4} \ve^{\mathcal{L}}_{\Omega_{2}}(Y)\\
& \lesssim \frac{r}{\dist(\xi, \partial \Omega_2)^{d+1}}  \int_{B(\xi,  \dist(\xi, \partial \Omega_2)/4) \cap \Omega_2}  \varepsilon_{\Omega_2}^{\mathcal{L}}(Y) \,d Y.
\end{align*}
If $z \in \partial \Omega_2$ such that $ \dist(\xi, \partial \Omega_2)=|z-\xi|$, the latter integral is bounded by a constant multiple of

\begin{align*}
\frac{1}{\dist(\xi, \partial \Omega_2)^{d}}  \int_{B(z,  2\dist(\xi, \partial \Omega_2)) \cap \Omega_2}  \varepsilon_{\Omega_2}^{\mathcal{L}}(Y) \,d Y  \lesssim 1,
\end{align*}
where the last inequality follows from \eqref{e:KP-Car}. \\

Assume now that $\xi \in  \partial \Omega_1 \cap \Omega_2$ and $r \in ( \frac{\dist(\xi, \partial \Omega_2)}{80}, \diam \Omega_1)$, and let $z \in \partial \Omega_2$ such that $ \dist(\xi, \partial \Omega_2)=|z-\xi|$. Now it is clear that $B(\xi, 2r) \subset B(z, 82 r)$ and arguing as before we can prove that \eqref{e:KP-Car} holds for $\Omega_2$. This concludes our proof since in the case $\xi \in  \partial \Omega_1 \cap \partial \Omega_2$ the result follows trivially.

\end{proof}

Recall now the following theorem.

\begin{theorem}[{\cite[Theorem 2.6]{KP01}}]
Let $\mathcal{L}=\div \mathcal{A}\grad$ be an elliptic operator satisfying the \hyperref[d:KP]{KP-condition} in $\Omega$ and let $\Omega\subset \bR^{d+1}$ be a bounded Lipschitz domain. Then the elliptic measure associated to $\mathcal{L}$ is in $A_{\infty}(\cH^{d}|_{\d\Omega})$. 
\end{theorem}


One can show that the same result holds in NTA domains with Ahlfors regular boundary. Indeed, if one uses \cite{KP01} instead of Dahlberg's result and Lemma \ref{l:Car2to1}, the arguments of \cite{DJ90} carry over to the elliptic case and give \Theorem{DJ}.


\section{The strong Markov property}

The aim of this section is to prove the following identity.

\begin{lemma} 
\label{l:identity-hm}

Let  $\Omega_1$ and  $\Omega_2$ be open subsets of $\R^{d+1}$ so that $\Omega_1 \subset \Omega_2$. Suppose every point in $\d_{\infty}\Omega_{1}$ is regular for $\Omega_{1}$ and every point in $\d\Omega_{2}\cap \d\Omega_{1}$ is regular for $\Omega_{2}$. If $\Omega_{1}$ is unbounded, also assume $\infty$ is regular for $\Omega_{2}$. If $E$ is a Borel subset of $\d \om_2$, then for all $X \in \om_1$,
\begin{equation}\label{e:identity-hm}
\omega^{\mathcal{L},X}_{\Omega_2} (E) = \omega^{\mathcal{L},X}_{\Omega_1} (E) +\int_{\partial
\Omega_1 \setminus \partial \Omega_2} \omega^{\mathcal{L},Y}_{\Omega_2} (E)  \,d \omega^{\mathcal{L},X}_{\Omega_1}.
\end{equation}
\end{lemma}

For the case of harmonic measure, this is well known (see for example  \cite{Bou87}) and follows from the strong Markov property of Brownian motion. Here, we supply an analytic proof that also works for elliptic measures.

\begin{proof}
We will drop the exponent $\mathcal{L}$ for easier reading. Let $E\subset \d\Omega_{2}$ be any compact set and let $\phi_{j}$ be a decreasing sequence of continuous compactly supported functions so that $0\leq \phi_{j}\leq 1$ and $\phi_{j}\downarrow \one_{E}$ pointwise everywhere. Let 
\[
u_{j}(X)=\int \phi_{j} d\omega_{\Omega_{2}}^{X}
\;\;\; \mbox{and}\;\;\; 
v_{j}(X)=\int u_{j} d\omega_{\Omega_{1}}^{X}.\]
We claim that
\begin{equation}\label{e:v=u}
v_{j}(X)=u_{j}(X) \;\; \mbox{ for all }X\in \Omega_{1}.\end{equation}
Indeed, since all points in $\d_{\infty}\Omega_{1}$ are regular, by \Lemma{sameboundaryval}, we need to show that 
\[
\lim_{X\rightarrow x} v_{j}(X) = \lim_{X\rightarrow x} u_{j}(X) \mbox{ for all }x\in \d_{\infty}\Omega_{1}.\]
Let $x\in \d_{\infty}\Omega_{1}$. Since $u_{j}$ is continuous in $\Omega_{2}$, it is continuous in $\d\cnj{\Omega_{1}}\cap \Omega_{2}$, and since $\phi_{j}$ is continuous and  every point in $\d_{\infty}\Omega_{1}$ is regular, $u_{j}$ extends continuously to $\cnj{\Omega_{2}}$ and thus also to $\cnj{\Omega_{1}}$. Thus, $v_{j}$ is also continuous in $\cnj{\Omega_{1}}$ and with the same boundary values except perhaps at $\infty$. Hence, we need only show that $u_{j}$ is continuous at $\infty$. For this, we just observe that $\lim_{\Omega_{2}\ni X\rightarrow\infty} u_{j}(X)=0$ for $i=1,2$ since $\infty$ is regular for $\Omega_{2}$, and so clearly $\lim_{\Omega_{1}\ni X\rightarrow\infty} u_{j}(X)=0$.  This proves \eqn{v=u}. Thus, for $X\in \Omega_{1}$,
\begin{align*}
u_{j}(X) 
 & = v_{j}(X)
 =\int_{\d\Omega_{1}}u_{j}d\omega_{\Omega_{1}}^{X}
   =\int_{\d\Omega_{1}\cap\Omega_{2}} u_{j}d\omega_{\Omega_{1}}^{X}
 + \int_{\d\Omega_{1}\cap\d\Omega_{2}} u_{j}d\omega_{\Omega_{1}}^{X}\\
 & = \int_{\d\Omega_{1}\cap\Omega_{2}} u_{j}d\omega_{\Omega_{1}}^{X}
 +\int_{\d\Omega_{1}\cap\d\Omega_{2}} \phi_{j}d\omega_{\Omega_{1}}^{X}.
 \end{align*}
 Since $\phi_{j}\downarrow \one_{E}$, by the monotone convergence theorem, $u_{j}(X)\downarrow \omega_{\Omega}^{X}(E)$ pointwise everywhere, and so also by the monotone convergence theorem twice (once with $u_{j}$ and once again with $\phi_{j}$)
  \begin{align*}
  \omega_{\Omega_{2}}^{X}(E)
&   = \lim_j \left( \int_{\d\Omega_{1}\cap\Omega_{2}} u_{j}\,d\omega_{\Omega_{1}}^{X}
 +\int_{\d\Omega_{1}\cap\d\Omega_{2}} \phi_{j}\,d\omega_{\Omega_{1}}^{X} \right)\\
&  = \int_{\d\Omega_{1}\cap\Omega_{2}} \omega_{\Omega_{2}}^{Y}(E) \,d\omega_{\Omega_{1}}^{X}(Y)
 +\omega_{\Omega_{1}}^{X}(E).
 \end{align*}
 
 This proves the lemma for $E\subset \d\Omega_{2}$ compact. Now let $E\subset \d\Omega_{2}$ be an arbitrary Borel set. Let $\{Y_{j}\}$ be a countable dense set in $\Omega_{1}$ so that $Y_{0}=X$. For $j\in \bN$, pick $E_{ij}\subset E$ compact so that $\omega_{\Omega_{1}}^{Y_{j}}(E\backslash E_{ij})<i^{-1}$ and $\omega_{\Omega_{2}}^{X}(E\backslash E_{i1})<i^{-1}$.  Then by continuity, $\omega_{\Omega_{1}}^{Y}(E\backslash \bigcup_{j} E_{ij})<i^{-1}$ for all $i$. Hence, if we enumerate the sets $\{E_{ij}\}=\{E^{k}\}$ and let $E_{k}=\bigcup_{\ell=1}^{k} E^{\ell}$, then each $E_{k}$ is compact and $\omega_{\Omega_{1}}^{Y}(E_{k})\rightarrow \omega_{\Omega_{1}}^{Y}(E)$ and $\omega_{\Omega_{2}}^{X}(E_{k})\rightarrow \omega_{\Omega_{2}}^{X}(E) $. We now apply the lemma to the compact set $E_{k}$ and use the monotone convergence theorem.

%
%

\end{proof}

\bibliographystyle{alpha}

\newcommand{\etalchar}[1]{$^{#1}$}
\def\cprime{$'$}

\end{document}